\documentclass[11pt,reqno]{amsart}

\usepackage{amsfonts}
\usepackage{amsmath,amscd}
\usepackage{tikz-cd}
\usepackage{amsthm}
\usepackage{galois}
\usepackage{chemarrow}
\usepackage[all]{xy}
\usepackage{graphicx}
\usepackage[colorlinks,
            linkcolor=blue,
            anchorcolor=black,
            citecolor=red
            ]{hyperref}	
\usepackage{mathrsfs}
\usepackage{extarrows}
\usepackage{texnames}

\def\End{{\rm End}}

\def\Ker{{\rm Ker}}

\def\rk{{\rm rk}}

\def\Tr{{\rm Tr}}

\def\Vol{{\rm Vol}}

\theoremstyle{plain}

\newtheorem{theorem}{Theorem}[section]

\newtheorem{proposition/example}[theorem]{Proposition/Example}
\newtheorem{proposition}[theorem]{Proposition}
\newtheorem{corollary}[theorem]{Corollary}
\newtheorem{lemma}[theorem]{Lemma}

\theoremstyle{definition}
\newtheorem{definition}[theorem]{Definition}
\newtheorem{remark}[theorem]{Remark}
\newtheorem{example}[theorem]{Example}

\newtheorem{conjecture/question}[theorem]{Conjecture/Question}

\newtheorem{remark/definition}[theorem]{Remark/Definition}
\newtheorem{definition/notation}[theorem]{Definition/Notation}

\numberwithin{equation}{section}

\begin{document}

\baselineskip=15pt

\title[Hitchin-Kobayashi Correspondence for Quiver Bundles]{The Hitchin--Kobayashi Correspondence for Quiver Bundles over Generalized K\"ahler Manifolds}

\author{Zhi Hu}

\address{{\it Zhi Hu}: \textsc{Research Institute of Mathematical Science, Kyoto University, Kyoto, Japan}
\endgraf \textsc{School of Mathematics, University of Science and Technology of China, Hefei 230026, China}}

\email{halfask@mail.ustc.edu.cn}

\author{Pengfei Huang}

\address{{\it Pengfei Huang}: \textsc{School of Mathematics, University of Science and Technology of China, Hefei 230026, China}
\endgraf \textsc{Laboratoire J.A. Dieudonn\'e, Universit\'e C\^ote d'Azur, CNRS, 06108 Nice, France}}

\email{pfhwang@mail.ustc.edu.cn; pfhwang@unice.fr}


\date{}

\begin{abstract}In this paper, we establish the Hitchin--Kobayashi correspondence for the  $I_\pm$-holomorphic  quiver bundle $\mathcal{E}=(E,\phi)$ over a compact generalized K\"{a}hler manifold $(X, I_+,I_-,g, b)$ such that $g$ is Gauduchon with respect to both $I_+$ and $I_-$, namely $\mathcal{E}$ is $(\alpha,\sigma,\tau)$-polystable if and only if $\mathcal{E}$ admits an  $(\alpha,\sigma,\tau)$-Hermitian--Einstein metric.
\end{abstract}

\maketitle
\tableofcontents
\section{Introduction}
 The Hitchin--Kobayashi correspondence exhibits a deep relation between the algebraic notion of stability and the existence of special metrics on holomorphic vector bundles. There are several generalizations for this correspondence along different directions. For example, one  replaces   base manifolds  with  Hermitian manifolds with Gauduchon metric \cite{LY} or non-compact K\"{a}hler manifolds satisfying some analytic conditions \cite{S1}; one generalizes Yang--Mills system to other gauge theoretic systems, such as introducing Higgs fields or vortex fields via dimensional  reduction \cite{H1,S2,B}, introducing singularities for Hermitian--Einstein connection and parabolic structure on vector bundle \cite{TM,P}; introducing frame structure via  vacuum expectation value of the scalar fields in $N= 2$ vector multiplet \cite{D}; one changes the stability condition, typically relaxes  to semistability and approximate Hermitian--Einstein metric \cite{C1,C2,LZ}; one considers an analog of such correspondence in positive characteristic or mixed characteristic \cite{F,LSZ}.

In present paper, our considerations focus on  generalized K\"ahler manifold as the base manifold  and quiver bundle as the  gauge theoretic system. Generalized K\"ahler manifold was first discovered by Gates, Hull, and Ro\v{c}ek as the target space of $N=(2,2)$ sigma model \cite{Gates}, and then reformulated under the context of Hitchin's generalized complex geometry \cite{H2,G1} by Gualtieri \cite{G2}. There are abundant candidates  for generalized K\"ahler manifold, for example, all degenerate del Pezzo surfaces and all Hirzebruch surfaces admits non-trivial  generalized K\"ahler structures \cite{H3}. On the other hand, quiver bundle coming from quiver gauge theory consists of  a set of vector bundles and a set of morphisms between these bundles \cite{Pr1,Pr2}.

We will define the notion of holomorphic quiver bundles over a generalized K\"ahler manifold, and introduce suitable stability and good metric for them. We should be faced with some new features in our setting: such stability depends on several real parameters reflecting the  generalized K\"ahler structure on base manifold and quiver structure on gauge theoretic system, and such metric satisfies a series of mutually coupled equations. Then we prove certain set-theoretic  Hitchin--Kobayashi correspondence, namely we have the following main theorem which generalizes the results in \cite{Pr1,Hu}.
\begin{theorem}[= Theorem \ref{thm3.15}] Let $Q=(Q_0,Q_1)$ be a quiver, and $\mathcal{E}=(E,\phi)$ be an $I_\pm$-holomorphic $Q$-bundle over an  $n$-dimensional compact  generalized K\"{a}hler manifold $(X, I_+,I_-,g, b)$ such that $g$ is Gauduchon with respect to both $I_+$ and $I_-$, then  $\mathcal{E}$ is $(\alpha,\sigma,\tau)$-polystable if and only if $\mathcal{E}$ admits an  $(\alpha,\sigma,\tau)$-Hermitian--Einstein metric.
\end{theorem}

 More related questions are proposed. Since a type  of interesting generalized K\"ahler manifolds, so-called generalized Calibi--Yau manifolds appear in compactification of Type II string theory, must be non-compact, we need generalize  such correspondence to the non-compact case.  The parameters in the definition of stability form a parameter space of stability conditions which is partitioned into chambers, studying the wall-crossing on this space is also an interesting topic, maybe the Hitchin--Kobayashi correspondence can play some role.

\section{Setups}
In this paper, a generalized K\"ahler manifold  refers to the geometric object defined by the following two equivalent approaches.

\begin{definition}(\cite{G1})A  manifold $X$ is called a {\em generalized K\"ahler manifold} if it carries two generalized complex structures $J_1,J_2\in \mathrm{End}(TX\oplus T^*X)$ satisfying
\begin{itemize}
  \item $J_1J_2=J_2J_1$,
  \item the symmetric pairing $G(A,B)=\langle J_1(A),J_2(B)\rangle$ is positive-definite for any non-zero $A,B\in TX\oplus T^*X$, where $\langle\cdot,\cdot\rangle$ denotes the natural inner product on $TX\oplus T^*X$.
\end{itemize}
\end{definition}

\begin{definition}(\cite{G2})
A manifold $X$ is called a {\em generalized K\"ahler manifold} if it carries the data $(I_+,I_-,g,b)$, where
\begin{itemize}
  \item $I_{\pm}$ are two complex structures on $X$,
  \item $g$ is a Riemannian metric on $X$,
  \item $b$ is a two-form on $X$,
  \item $I_{\pm}$ are   parallel with respect to the connections $\nabla^{\pm}=\nabla\pm\frac{1}{2}g^{-1}\mathbb{H}$, respectively, where $\nabla$ is the Levi-Civita connection of $g$ and $\mathbb{H}=db$.
\end{itemize}
\end{definition}
 The  generalized Calabi--Yau manifold is an important  kind of generalized K\"ahler manifold.
\begin{definition}(\cite{Z})A {\em generalized Calabi--Yau manifold} is a generalized K\"ahler manifold $(X,J_1,J_2)$ such that both nowhere vanishing pure spinors $\psi_1,\psi_2$ corresponding to $J_1, J_2$, respectively satisfy the following conditions
 \begin{itemize}
   \item $d\psi_1=d\psi_2=0$,
   \item $(\psi_1,\bar\psi_1)=(\psi_2,\bar\psi_2)$,
 \end{itemize}
where  $(\bullet,\bullet)$ is the Mukai pairing.

\begin{remark}
More generally, one  defines the twisted  generalized K\"ahler manifold as the manifold $X$ with  4-tuple $(I_+,I_-,g,\mathbb{H})$, where $I_{\pm},g$ is the same as above, and $\mathbb{H}$ is a closed 3-form such that $I_{\pm}$ are parallel with respect to the connections $\nabla^{\pm}=\nabla\pm\frac{1}{2}g^{-1}\mathbb{H}$, respectively.  Similarly, one can also introduce the twisted generalized Calabi--Yau manifold by replacing the first condition on pure spinors  by $d_\mathbb{H}\psi_1=d_\mathbb{H}\psi_2=0$ for $d_\mathbb{H}=d+\mathbb{H}\wedge $. When $\mathbb{H}$ is an exact 3-form, they reduce to  the generalized K\"ahler manifold and generalized Calabi--Yau manifold defined as above.
\end{remark}

\end{definition}
\begin{definition}(\cite{Hu}) Let $(X,I_+,I_-,g,b)$ be a  generalized K\"ahler manifold, and $E$ be a complex vector bundle over $X$.
 $E$ is called {\em $I_{\pm}$-holomorphic} if there are two operators $\bar \partial _{\pm}:C^\infty (E)\rightarrow C^\infty(E\otimes T^{0,1}_{I_{\pm}}X)$ such that they define a holomorphic structure on $E$ with respect to $I_{\pm}$ respectively.
\end{definition}
Given an $I_{\pm}$-holomorphic vector bundle $(E, \bar\partial_{+},\bar\partial_-)$, denote by $\iota$ the natural isomorphisms between $\bar L_{\pm}$ and $T^{0,1}_{I_{\pm}}X$, one defines $\bar D_{\pm, s}(v):=\bar \partial _{\pm,\iota(s)}(v)$ for $s\in C^\infty(\bar L_{\pm})$ and $v\in C^\infty(E)$, where $L_{+}=L_1\cap L_2$, $L_-=L_1\cap \bar L_2$ with $L_1$, $L_2$
be $\sqrt {-1}$-eigensubbundles of $(TX\oplus T^*X)\otimes\mathbb{C}$ with respect to the generalized complex structure $J_{1,2}$ determined by
$$J_{1,2}=\frac{1}{2}\left(
            \begin{array}{cc}
              I_+\pm I_- & -(\omega^{-1}_+\mp\omega_-^{-1}) \\
             \omega_+\mp\omega_- & -(I_+^*\pm I_-^*) \\
            \end{array}
          \right)
$$
for K\"ahler forms $\omega_{\pm}=g(I_{\pm}\cdot,\cdot)$. Then $\bar D=\bar D_++\bar D_-: C^\infty(E)\rightarrow C^\infty(E\otimes \bar L_1)$ defines a generalized holomorphic bundle with respect to $J_1$ if and only if $\bar \partial_+\bar\partial_-+\bar \partial_-\bar\partial_+=0$ \cite{Hu}.

Moreover, we  make the following assumptions on the $n$-dimensional  generalized K\"ahler manifold $(X, I_+,I_-,g, b)$ in this paper:
\begin{itemize}
\item$g$ is Gauduchon, i.e., $dd^c_{\pm}\omega_{\pm}^{n-1}=0$ and $d\mathrm{Vol}_g=\frac{1}{n!}\omega_{\pm}^n$, where $d^c_{\pm}=I_{\pm}\circ d\circ I_{\pm}$;
\item $X$ is compact.
\end{itemize}

The first assumption is not too restrictive. It can be satisfied for generalized K\"ahler 4-manifolds automatically, and for real compact Lie groups. On the second assumption, we have the following no-go type theorem.

\begin{proposition}
\begin{enumerate}
\item[(1)] A compact twisted generalized K\"{a}hler surface has even first Betti number if $\mathbb{H}$ is exact, and has odd first Betti number if $\mathbb{H}$ is not exact.
 \item[(2)] A compact twisted generalized Calabi--Yau manifold must be a usual Calabi--Yau manifold.
                   \end{enumerate}
\end{proposition}
\begin{proof} (Sketch) The first result has been  proved by the authors of \cite{AG}. We only prove the second claim. The structure of generalized Calabi--Yau reduces the structure group $O(2n,2n)$ of $TX\oplus T^*X$ to $SU_+(n)\times SU_-(n)$, then  there are two globally defined $SU_\pm(n)$-invariant spinors $\xi_\pm$. The  constraints on pure spinors  can be rewritten in terms of $\xi_\pm$ \cite{W}
\begin{align*}
 (\nabla_M\pm \frac{1}{4} M\lrcorner \mathbb{H})\cdot\xi_{\pm}=&0,\\
 (df\pm \frac{1}{2}\mathbb{H})\cdot\xi_{\pm}=&0
\end{align*}
 for  $ \forall M\in C^\infty(TX)$, exact three-form $\mathbb{H}=db$ and smooth function $f=\log\frac{1}{(\psi_1,\bar\psi_1)}$, where  $\nabla$ denotes the spin connection with respect to $g$, and $\cdot$ stands for the Clifford multiplication. We only need to show if $X$ is compact then $\mathbb{H}$ vanishes. Indeed, the following equations are derived from the above conditions \cite{Hs}
 \begin{align*}
   R^{(g)}_{\mu\nu}-\frac{1}{4}\mathbb{H}_{\mu\alpha\beta}\mathbb{H}_{\nu\gamma\delta}g^{\alpha\delta}g^{\beta\delta}+2\nabla_\mu\nabla_\nu f&=0,\\
   g^{\mu\alpha}\nabla_\mu(e^{-2f}\mathbb{H}_{\alpha\beta\gamma})&=0,\\
   R^{(g)}+4g^{\mu\nu}\nabla_\mu\nabla_\nu f-4g^{\mu\nu}\nabla_\mu\nabla_\nu f-\frac{1}{12}g^{\mu\alpha}g^{\nu\beta}g^{\lambda\gamma}\mathbb{H}_{\mu\nu\lambda}\mathbb{H}_{\alpha\beta\gamma}&=0.
 \end{align*}
 After taking trace  we get $g^{\mu\nu}\nabla_\mu\nabla_\nu e^{-2f}-\frac{1}{6}e^{-2f}g^{\mu\alpha}g^{\nu\beta}g^{\lambda\gamma}\mathbb{H}_{\mu\nu\lambda}\mathbb{H}_{\alpha\beta\gamma}=0$, then integrating over $X$ implies the vanishing of  $\mathbb{H}$ if $X$ is compact.
\end{proof}

Now let $(E,\bar{\partial}_+,\bar{\partial}_-)$ be an $I_\pm$-holomorphic bundle over a generalized K\"ahler manifold $X$, fix a Hermitian metric $H$ on $E$, then there is a unique Chern-connection compatible with  the complex structures $I_\pm$ respectively,  given by
  $D_H^\pm: =\partial_H^\pm+\overline{\partial}_\pm$,
whose curvature form is denoted by $\mathbb{F}_H^\pm$. Then we define
the  degrees associated to the two Chern connections as follows:
\begin{equation*}
  \deg_\pm(E):=\frac{\sqrt{-1}}{2\pi}\int_X\Tr(\mathbb{F}_H^\pm)\wedge\omega_\pm^{n-1},
\end{equation*}
which are independent of the choice of Hermitian metric $H$ on $E$, since for any two Hermitian metrics $H$ and $H^\prime$ on $E$, we have
$\Tr(\mathbb{F}^\pm_H)=\Tr(\mathbb{F}_{H^\prime}^\pm)+\partial_\pm\overline{\partial}_\pm(\log\det((H^\prime)^{-1}H))$.

\begin{definition}
\begin{enumerate}
\item[(1)] A {\em quiver} $Q=(Q_0,Q_1,h,t: Q_1\to Q_0)$ is a $4$-tuple, where
                   \begin{itemize}
                     \item $ Q_0$ and $Q_1$ are finite sets of vertices and arrows, respectively,
                     \item $h,t: Q_1\to Q_0$ map each arrow $a\in Q_1$ to its head $h(a)$ and tail $t(a)$, respectively.
                   \end{itemize}
\item[(2)] A {\em $Q$-sheaf} on a complex manifold $X$ is a pair $\mathcal{E}=(E,\phi)$, where $E=\{E_i\}_{i\in Q_0}$ is a collection of sheaves of $\mathcal{O}_X$-modules and $\phi=\{\phi_a\}_{a\in Q_1}$ a collection of morphisms $\phi_a:  E_{t(a)}\to E_{h(a)}$. In particular, if each $E_i$ is locally free, $\mathcal{E}$ is called a $Q$-bundle. A $Q$-subsheaf of $\mathcal{E}=(E,\phi)$ is a $Q$-sheaf $\mathcal{E}^\prime=(E^\prime,\phi^\prime)$ such that $E^\prime_i$ is a subsheaf of $E_i$ for each vertex $i$ and $\phi^\prime_a=\phi_a|_{E_{h(a)}}$ for each arrow $a$.
\item[(3)] A {\em morphism $f: \mathcal{E}\to \mathcal{F}$ between two  $Q$-sheaves} $\mathcal{E}=(E,\phi)$ and $\mathcal{F}=(F,\varphi)$ is a collection of  morphisms $f_i: E_i\to F_i$ such that for each arrow $a\in Q_1$, the following diagram commutes:
$$\begin{CD}
  E_{t(a)} @>\phi_a>> E_{h(a)} \\
  @V  f_{t(a)} VV @V f_{h(a)} VV  \\
  F_{t(a)} @>\varphi_a>> F_{h(a)}.
\end{CD} $$
\item[(4)] A {\em Hermitian metric} on a $Q$-bundle $\mathcal{E}=(E,\phi)$ is a collection $H=\{H_i\}_{i\in Q_0}$ of Hermitian metrics $H_i$ on $E_i$.  For each arrow $a\in Q_1$, by virtue of the Hermitian metrics at tail and head, the morphism $\phi_a$ has a smooth adjoint $\phi_a^{*H}: E_{h(a)}\to E_{t(a)}$ with respect to the Hermitian metrics at tail and head, that is, $H_{h(a)}(\phi_a(u), v)=H_{t(a)}(u,\phi_a^{*H}(v))$ for any sections $u,v$ of $E_{h(a)},E_{t(a)}$.
   \item[(5)]  A {\em $Q$-bundle $\mathcal{E}=(E,\phi)$    on a generalized K\"ahler manifold} $(X, I_+,I_-,g, b)$ is called $I_{\pm}$-holomorphic if
   \begin{itemize}
     \item each $E_i$, $i\in Q_0$, is an $I_{\pm}$-holomorphic bundle, i.e., $E_i$ carries two holomorphic structures $\bar{\partial}_{+i},\bar{\partial}_{-i}$ with respect to $I_{\pm}$, respectively,
     \item each $\phi_a$, $a\in Q_1$, is $I_{\pm}$-holomorphic, namely
     \begin{equation*}
  \left\{\begin{split}
    \phi_a\circ\bar{\partial}_{+t(a)}=\bar{\partial}_{+h(a)}\circ\phi_a,\\
    \phi_a\circ\bar{\partial}_{-t(a)}=\bar{\partial}_{-h(a)}\circ\phi_a.
  \end{split}\right.
\end{equation*}
   \end{itemize}
   \item[(6)] A {\em morphism  $f: \mathcal{E}\to \mathcal{F}$ between two  $I_{\pm}$-holomorphic $Q$-bundles} $\mathcal{E}=(E,\phi)$ and $\mathcal{F}=(F,\varphi)$ is a collection of $I_\pm$-holomorphic morphisms $f_i: (E_i,\bar{\partial}_{i+},\bar{\partial}_{i-})\to (F_i,\bar{\partial}_{+i}',\bar{\partial}_{-i}')$, such that for each arrow $a\in Q_1$, the following diagram commutes:
$$\begin{CD}
  (E_{t(a)},\bar{\partial}_{+t(a)},\bar{\partial}_{-t(a)}) @>\phi_a>> (E_{h(a)},\bar{\partial}_{+h(a)},\bar{\partial}_{-h(a)}) \\
  @V f_{t(a)} VV @V f_{h(a)} VV  \\
  (F_{t(a)},\bar{\partial}_{+t(a)}',\bar{\partial}_{-t(a)}') @>\varphi_a>> (F_{h(a)},\overline{\partial}_{+t(a)}',\bar{\partial}_{-t(a)}').
\end{CD} $$
\item[(7)]  An $I_\pm$-holomorphic $Q$-bundle $\mathcal{E}=(E,\phi)$ is said to be {\em simple} if any endomorphism $f: \mathcal{E}\rightarrow \mathcal{E}$ must have the form  $f=\{c \mathrm{Id}_{E_i}\}_{i\in Q_0}$ for a constant $c\in\mathbb{C}$.
     \end{enumerate}
 \end{definition}

\begin{definition}(\cite{Pr1,Hu}) Let $\mathcal{E}=(E,\phi)$ be an  $I_{\pm}$-holomorphic $Q$-bundle.
 \begin{enumerate}
  \item[(1)] A {\em coherent $Q$-subsheaf} $\mathcal{F}$ of  $\mathcal{E}$ is a 4-tuple $\mathcal{F}=(\mathcal{F}_+,\mathcal{F}_-, \mathbb{S}_+,\mathbb{S}_-)$, where
        \begin{itemize}
          \item $\mathcal{F}_\pm=(F_\pm,\varphi)$ are $Q$-subsheaves of the $Q$-sheaves $\mathcal{E}_\pm$, where $\mathcal{E}_\pm=(E_\pm=\{E_{\pm i}=(E_i,\bar\partial_{\pm i})\}_{i\in Q_0},\phi=\{\phi_a\}_{a\in Q_1})$, respectively,
          \item $\mathbb{S}_\pm=\{S_{\pm i}\}_{i\in Q_0}$  are collections of  analytic subsets of $(X,I_\pm)$,respectively, such that for each $i\in Q_0$
          \begin{itemize}
            \item $S_i=S_{+i}\cup S_{-i}$ has codimension at least 2,
            \item $F_{\pm i}|_{X\backslash S_{\pm i}}$ are locally free and $F_{+i}|_{X\backslash S_i}=F_{- i}|_{X\backslash S_i}:=F_i$ as smooth vector bundles.
          \end{itemize}\end{itemize}
   \item[(2)] For any  coherent subsheaf $\mathcal{F}$ of   $\mathcal{E}$, we define
 {\em $(\alpha,\sigma,\tau)$-degree} and {\em $(\alpha,\sigma,\tau)$-slope}  as follows:
\begin{align*}
  \deg_{\alpha,\sigma,\tau}(\mathcal{F}):=&\sum\limits_{i\in Q_0}\alpha_i\sigma_i\deg_+({F_{+i}})+\sum\limits_{i\in Q_0}(1-\alpha_i)\sigma_i\deg_-({F_{-i}})-\sum_{i\in Q_0}\tau_i \rk(F_i),\\
\mu_{\alpha,\sigma,\tau}(\mathcal{F}):=&\frac{\sum\limits_{i\in Q_0}\alpha_i\sigma_i\deg_+({F_{+i}})}{\sum\limits_{i\in Q_0}\sigma_i\rk({F_{i}})}+\frac{\sum\limits_{i\in Q_0}(1-\alpha_i)\sigma_i\deg_-({F_{-i}})}{\sum\limits_{i\in Q_0}\sigma_i\rk({F_{i}})}-\frac{\sum\limits_{i\in Q_0}\tau_i\rk({F_{i}})}{\sum\limits_{i\in Q_0}\sigma_i\rk({F_{i}})},
\end{align*}
where $\alpha_i\in(0,1),\sigma_i\in \mathbb{R}^+,\tau_i\in\mathbb{R}$, and $\rk(F_i)=\rk(F_{+i})=\rk(F_{-i})$ denotes the rank of the corresponding sheaves.
 $\mathcal{E}$ is called {\em $(\alpha,\sigma,\tau)$-stable} (respectively, {\em $(\alpha,\sigma,\tau)$-semistable}) if for any proper coherent $Q$-subsheaf $\mathcal{F}$, we have
$\mu_{\alpha,\sigma,\tau}(\mathcal{F})<\mu_{\alpha,\sigma,\tau}(\mathcal{E})$ (respectively, $\mu_{\alpha,\sigma,\tau}(\mathcal{F})\leq
\mu_{\alpha,\sigma,\tau}(\mathcal{E}))$,
and $\mathcal{E}$ is called {\em polystable} if it is the the direct sum of $(\alpha,\sigma,\tau)$-stable $Q$-subsheaves of the same slope with $\mathcal{E}$.
 \end{enumerate}
 \end{definition}
Due to the classical extension theorem \cite{BA}, we have the following extension theorem for the coherent $Q$-subsheaves.

\begin{proposition}\label{prop2.8}
For each $i\in Q_0$, there are unique holomorphic bundles $\hat F_{\pm i}$ over $(X, I_\pm)$ extending  the bundles $F_{\pm i}|_{X\backslash S_{\pm i}}$, respectively, hence there is a unique $I_\pm$-bundle $(\hat F_i,\hat{\bar \partial}_+,\hat{\bar\partial}_-)$ over  $(X, I_\pm)$ extending the $I_\pm$-bundle $(F_i,\bar\partial_+,\bar\partial_-)$ over $X\backslash S_i$.
 \end{proposition}

The following facts are  analogs of classical cases \cite{Koba}.
\begin{proposition}\label{prop2.9} 
Let $f: \mathcal{E}\to\mathcal{F}$ be a morphism between two $I_\pm$-holomorphic $Q$-bundles $\mathcal{E}=(E,\phi)$ and $\mathcal{F}=(F,\varphi)$.
\begin{enumerate}
  \item[(1)]  If $\mathcal{E}$ and $\mathcal{F}$ are $(\alpha,\sigma,\tau)$-semistable, then
$\mu_{\alpha,\sigma,\tau}(\mathcal{E})\leq\mu_{\alpha,\sigma,\tau}(\mathcal{F})$.
  \item[(2)] If $\mathcal{E}$ and $\mathcal{F}$ are stable of the same $(\alpha,\sigma,\tau)$-slope, then $f$ is an isomorphism.
  \item[(3)] If $\mathcal{E}$ is $(\alpha,\sigma,\tau)$-stable, then it is simple.
\end{enumerate}
\end{proposition}

\begin{definition}\label{def2.10}
A Hermitian metric $H$ on an $I_\pm$-holomorphic $Q$-bundle $\mathcal{E}=(E,\phi)$ is called
an  {\em $(\alpha,\sigma,\tau)$-Hermitian--Einstein metric} if for each vertex $i\in Q_0$ it satisfies the  following equations
\begin{align*}
 &\sqrt{-1}(\alpha_i\sigma_i\mathbb{F}^+_{H_i}\wedge \omega_+^{n-1}+(1-\alpha_i)\sigma_i\mathbb{F}^-_{H_i}\wedge \omega_-^{n-1})\\&+(n-1)!\left(\sum_{a\in h^{-1}(i)}\phi_a\circ \phi_a^{*H}-\sum_{a\in t^{-1}(i)} \phi_a^{*H}\circ\phi_a\right) \Vol_g\\
 =&(n-1)!\lambda(\tau_i+\gamma\sigma_i)\mathrm{Id}_{E_i}\Vol_g
\end{align*}
with constants $\lambda=\frac{2\pi}{(n-1)! \int_Xd\Vol_g}$ and $\gamma$.
\end{definition}

\begin{remark}
Taking trace and  the sum over all vertices and then doing integral over $X$ on  both sides, we see that $\gamma$ is exactly the slope $\mu_{(\alpha,\sigma,\tau)}(\mathcal{E})$.
\end{remark}

We employ the following notations:
\begin{itemize}
  \item $S(E_i, H_i)$ is the space of smooth $H_i$-Hermitian endomorphisms of $E_i$, $S^+(E_i, H_i)\subset S(E_i, H_i)$ is the open subset of positive-definite ones;
  \item $S(\mathcal{E},H)=\prod\limits_{i\in Q_0}S(E_i, H_i)$, $S^+(\mathcal{E},H)=\prod\limits_{i\in Q_0}S^+(E_i, H_i)$. The metric $H$ induces a metric on  $S(\mathcal{E},H)$, also denoted by $H$,  namely $\langle f,g\rangle_H=\sum_{i\in Q_0}\langle f_i,g_i\rangle_{H_i}$ for $f=(f_i)_{i\in Q_0},g=(g_i)_{i\in Q_0}\in S(\mathcal{E}, H)$.
      \item $L^p_k(S)$ denotes the corresponding Sobolev space.
      \item The pointwise  or global norms and inner products $|\bullet|, \langle\bullet,\bullet\rangle, ||\bullet||,\langle\langle\bullet,\bullet\rangle\rangle_{L^2}$ are defined with respect to the metrics $H_i$ or induced metric induced metrics on $E_{h(a)}\otimes (E_{t(a)})^*$ from the  metrics $H_{h(a)}$ and $H_{t(a)}$  unambiguously depending on the contexts.
\end{itemize}

\begin{proposition}\label{prop2.12}
Let  $H$ be an $(\alpha,\sigma,\tau)$-Hermitian--Einstein metric on an  $I_\pm$-holomorphic $Q$-bundle $\mathcal{E}=(E,\phi)$ over $X$, then we define
\begin{align*}
 C_1(E_i)=&\int_X\Tr (\sqrt{-1}(\alpha_i\Lambda_+ \mathbb{F}^+_{H_i}+(1-\alpha_i)\Lambda_-\mathbb{F}^-_{H_i}))d\Vol_g,\\
 C_2(E_i)=&\int_X\left(\Tr(( \alpha_i\mathbb{F}_{H_i}^+)^2)\wedge \frac{\omega^{n-2}_+}{(n-2)!}+\Tr(((1-\alpha_i) \mathbb{F}_{H_i}^-)^2)\wedge \frac{\omega^{n-2}_-}{(n-2)!}\right),\\
 C^\prime_2(E_i)=&\langle\langle\alpha_i\Lambda_+ \mathbb{F}^+_{H_i},(1-\alpha_i)\Lambda_- \mathbb{F}^-_{H_i}\rangle\rangle_{L^2},
\end{align*}
where $\Lambda_\pm$ is the adjoint of the operator of the wedge by $\omega_\pm$  with respect to the metric $g$.
When $\alpha_i=\alpha_j=\alpha$ for $\forall i,j\in Q_0$, the following inequality holds
\begin{align*}
  \sum_{i\in Q_0}\sigma_i(C_2(E_i)-C^\prime_2(E_i))+2\lambda\sum_{i\in Q_0}(\tau_i+\mu_{(\alpha,\sigma,\tau)}(\mathcal{E})\sigma_i)C_1(E_i)\geq0.
\end{align*}
\end{proposition}

\begin{proof}By assumption we have
\begin{align*}
 0=&\sum_{i\in Q_0}\frac{1}{\sigma_i}\Big|\Big|\sqrt{-1}(\alpha\sigma_i\Lambda_+ \mathbb{F}^+_{H_i}+(1-\alpha)\sigma_i\Lambda_-\mathbb{F}^-_{H_i})+\sum_{a\in h^{-1}(i)}\phi_a\circ \phi_a^{*H}-\sum_{a\in t^{-1}(i)} \phi_a^{*H}\circ\phi_a\\
 &\ \ \ \ \ \ \ \ \ \ \  -\lambda(\tau_i+\mu_{\alpha,\sigma,\tau}(\mathcal{E})\sigma_i)\mathrm{Id}_{E_i}\Big|\Big|^2_{L^2}\\
 =&\sum_{i\in Q_0}\Bigg(\alpha^2\sigma_i||\Lambda_+ \mathbb{F}^+_{H_i}||^2_{L_2}+(1-\alpha)^2\sigma_i||\Lambda_- \mathbb{F}^-_{H_i}||^2_{L_2}+2\alpha(1-\alpha)\sigma_i\langle\langle\Lambda_+ \mathbb{F}^+_{H_i},\Lambda_- \mathbb{F}^-_{H_i}\rangle\rangle_{L^2}\\
 &\ \ \ \ \ \ \ \ \ +\frac{1}{\sigma_i}\bigg|\bigg|\sum_{a\in h^{-1}(i)}\phi_a\circ \phi_a^{*H}-\sum_{a\in t^{-1}(i)} \phi_a^{*H}\circ\phi_a
   -\lambda(\tau_i+\mu_{\alpha,\sigma,\tau}(\mathcal{E})\sigma_i)\mathrm{Id}_{E_i}\bigg|\bigg|^2_{L^2}\Bigg)\\
   &-2\lambda\sum_{i\in Q_0}{(\tau_i+\mu_{(\alpha,\sigma,\tau)}(\mathcal{E})\sigma_i)}\int_X\Tr \bigg(\sqrt{-1}\Big(\alpha\Lambda_+ \mathbb{F}^+_{H_i}+(1-\alpha)\Lambda_-\mathbb{F}^-_{H_i}\Big)\bigg)d\Vol_g\\
   &+2\sum_{a\in Q_1}\mathrm{Re}\langle\langle\phi_a,[\sqrt{-1}(\alpha\Lambda_+ \mathbb{F}^+_{H}+(1-\alpha)\Lambda_-\mathbb{F}^-_{H}), \phi]_a\rangle\rangle_{L^2},
\end{align*}
where $[A,\phi]_a=A_{h(a)}\circ\phi_a-\phi_a\circ A_{t(a)}$ for $A\in \End(\mathcal{E})$.
Then we find the desired inequality by virtue of the following identities
\begin{align*}
  ||\Lambda_\pm \mathbb{F}_{H_i}^\pm||_{L^2}^2=& || \mathbb{F}_{H_i}^\pm||_{L^2}^2-\int_X\Tr(( \mathbb{F}_{H_i}^\pm)^2)\wedge \frac{\omega^{n-2}_\pm}{(n-2)!},\\
  \langle\langle\phi_a,[\sqrt{-1}\Lambda_\pm \mathbb{F}_{H}^\pm,\phi]_a\rangle\rangle_{L^2}=&\langle\langle\phi_a,2\sqrt{-1}\Lambda_\pm\bar\partial_\pm\partial^\pm_{H}\phi_a
  -(\partial^\pm_{H})^{* H,g}\partial^\pm_{H}\phi_a\rangle\rangle_{L^2}\\
  &+\langle\langle\phi_a,\frac{\sqrt{-1}}{(n-1)!}\star_g(\bar\partial_\pm(\omega_\pm^{n-1})\wedge \partial^\pm_{H}\phi_a)\rangle\rangle_{L^2}\\
  =&||\partial^\pm_{H}\phi_a||^2_{L^2}-\frac{\sqrt{-1}}{(n-1)!}\int_X\bar\partial_\pm|\phi_a|^2_H\wedge \partial_\pm(\omega_\pm^{n-1})\\
   =&||\partial^\pm_{H}\phi_a||^2_{L^2},
\end{align*}
where $\star_g$ denotes the Hodge star with respect to $g$,  the connections acting  on $\phi_a$ are the induced connections on $E_{h(a)}\otimes (E_{t(a)})^*$, and the $I_\pm$-holomorphicity of $\phi_a$'s plays a crucial roal in the second identity.
\end{proof}

We end this section with some examples.
\begin{example}\begin{enumerate}
\item We first consider $X=\mathbb{P}^1$ with the standard  K\"{a}hler structure $(I,\omega)$, it can be retreated as a generalized  K\"{a}hler manifold by taking  $I=I_+=I_-, \omega=\omega_+=\omega_-$.
                     Let $Q=(Q_0,Q_1,h,t)$ be a quiver with $Q_0=\{i,j\}$, $Q_1=\{a\}$ and $t(a)=i,h(a)=j$, then we consider the  $I_{\pm}$-holomorphic $Q$-bundle $\mathcal{E}=(E,\phi)$ over $\mathbb{P}^1$  given by $E_i=\mathcal{O}(m_i), E_j=\mathcal{O}(m_j)$ for $m_j\geq m_i$, and $0\neq\phi_a\in H^0(\mathbb{P}^1, \mathcal{O}(m_j-m_i))$. Obviously, $\deg_{\pm}\mathcal{O}(m)=m$, hence for  the  stability parameters $\alpha_i, \alpha_j$; $\sigma_i,\sigma_j$ and $\tau_i,\tau_j$, $\mathcal{E}$ is $(\alpha,\sigma,\tau)$-stable if and only if the following inequality holds
                     $$\sigma_i\sigma_j(m_j-m_i)< \sigma_i\tau_j-\sigma_j\tau_i.$$
In particular, the parameters $\sigma,\tau$ are subject to the condition
$$\sigma_i\tau_j-\sigma_j\tau_i>0,$$
which gives the constraints on these parameters as follows:
\begin{itemize}
\item if $\tau_i=0$, then $\tau_j>0$;
\item if $\tau_j=0$, then $\tau_i<0$;
\item if $\tau_i, \tau_j\neq0$, then $\frac{\sigma_i}{\sigma_j}>\frac{\tau_i}{\tau_j}$.
\end{itemize}

\item Now we consider the example of Hopf surfaces, which can be found in \cite{G2} (Example 1.21) and \cite{Hu} (Section 4 for details). Let $X$ be a standard Hopf surface, namely $X= \mathbb{C}^2\backslash\{(0,0)\}/(2(z_1,z_2)\sim (z_1,z_2))$, then $X$ is diffeomorphic to $S^3\times S^1$. Denote by $I_+$ the induced  complex structure from $\mathbb{C}^2$, the Hermitian  metric is given by 
 $$g=\frac{1}{4\pi |z|^2}(dz_1d\bar z_1+dz_2d\bar z_2),$$ 
for $|z|^2=z_1\bar z_1+z_2\bar z_2$, and the associated 2-form $\omega_+=gI_+$ is 
$$\omega_+=\frac{\sqrt{-1}}{4\pi |z|^2}(dz_1\wedge d\bar z_1+dz_2\wedge d\bar z_2).$$ 
One can specify another complex structure $I_-$ by providing a generator $$\Omega=\frac{1}{|z|^4}(\bar z_1 dz_1+z_2d\bar z_2)\wedge (\bar z_1dz_2-z_2d\bar z_1)$$ for $\Omega^{2,0}((X,I_+))$. It is easy to check that  $(g, I_-)$ is also Hermitian, and the associated 2-form is given by
$$\omega_-=\frac{\sqrt{-1}}{4\pi |z|^4}((\bar z_1dz_1+z_2d\bar z_2)\wedge(z_1d\bar z_1+\bar z_2dz_2)+(\bar z_1dz_2-z_2d\bar z_1)\wedge(z_1d\bar z_2-\bar z_2dz_1)).$$ Then $(I_+,I_-, g, \mathbb{H})$ defines a twisted generalized K\"{a}hler structure on $X$, where $\mathbb{H}=d^c_+\omega_+=-d^c_-\omega_-$ \cite{G2,Hu}. Actually, the torsion of twisted generalized K\"{a}hler structures on $X$ cannot be exact \cite{G2}. There is a natural projection $\mathrm{pr}:X\rightarrow \mathbb{P}^1$ onto $\mathbb{P}^1$ via $(z_1,z_2)\mapsto [z_1:z_2]$, and this projection is holomorphic with respect to $I_+$. We set $\mathcal{O}_+(m):=\mathrm{pr}^*\mathcal{O}_{\mathbb{P}^1}(m)$ for all $m\in\mathbb{Z}$, where $\mathcal{O}_{\mathbb{P}^1}(m)$ denotes the holomorphic line bundle on $\mathbb{P}^1$ of degree $m$. Consider the inverse map $\varrho:X\rightarrow X$, $(z_1,z_2)\mapsto (z_1,z_2)^{-1}:=\frac{1}{|z|^2}(\bar{z}_1,-z_2)$, which is a biholomorphic map from $(X,I_-)$ to $(X,I_+)$, and we introduce $\mathcal{O}_-(m):=\varrho^*\mathcal{O}_+(m)$ for all $m\in\mathbb{Z}$. For simplicity, we denote $\mathcal{O}_\pm(0)$ by $\mathcal{O}_\pm$. By Proposition 4.5 of \cite{Hu}, $\mathcal{O}_{+}(m)$  can be made into  an $I_{\pm}$-holomorphic line bundle $L_+(m):=(\mathbb{O}, \bar\partial_{m,+}, \bar\partial_{m,-})$ on $(X, g,I_+,I_-,\mathbb{H})$ such that $(\mathbb{O},\bar\partial_{m,+})\simeq \mathcal{O}_{+}(m)$ and $(\mathbb{O},\bar\partial_{m,-})\simeq \mathcal{O}_{-}(-m)$, where $\mathbb{O}$ denotes the topologically trivial line bundle $X\times\mathbb{C}$ on $X$. Similarly,  the  $I_{\pm}$-holomorphic line bundle associated to $\mathcal{O}_{-}(m)$ is denoted by $L_-(m):=(\mathbb{O},\bar\partial^\prime_{m,+}, \bar\partial^\prime_{m,-})$ with isomorphisms $(\mathbb{O},\bar\partial^\prime_{m,+})\simeq \mathcal{O}_+(-m)$ and $(\mathbb{O},\bar\partial^\prime_{m,-})\simeq \mathcal{O}_-(m)$. Moreover, one can show that   \cite{Hu}
                     \begin{align*}
                       \deg_{+}L_+(m)&=m,\qquad \deg_{-}L_+(m)=-m,\\
                       \deg_{+}L_-(m)&=-m,\quad\, \deg_{-}L_-(m)=m.
                     \end{align*}

Next we take the quiver $Q$ be the same as in (1), and an  $I_{\pm}$-holomorphic $Q$-bundle $\mathcal{E}=(E,\phi)$ which is given by $E_i=L_+(m_i), E_j=L_+(m_j)$ and $\phi_a$, where $\phi_a$ must vanish if $m_i\neq m_j$. Assume $m_i=m_j=m$ and $\phi_a$ is non-zero, then $\mathcal{E}$ is $(\alpha,\sigma,\tau)$-stable if and only if
 $$2m\sigma_i\sigma_j(\alpha_j-\alpha_i)<\sigma_i\tau_j-\sigma_j\tau_i.$$
                     
Finally, as the Example 4.11 in \cite{Hu}, let $V$ be a fixed smooth complex vector bundle of rank 2, we choose  $I_{\pm}$-holomorphic structures $\bar{\partial}_\pm^V$ on $V$ as follows:
\begin{itemize}
\item $\bar \partial_+^V$ is $I_+$-holomorphic structure such that $V_+:=(V,\bar \partial_+^V)$ is not isomorphic to a sum of two line bundles and is given by the non-trivial extension
$$0\rightarrow \mathcal{O}_+ \xrightarrow{\chi_+}V_+\rightarrow \mathcal{O}_+(-m_+)\rightarrow 0$$
for $m_+\in\mathbb{Z}^{>0}$,
\item $\bar \partial_-^V$ is $I_-$-holomorphic structure such that $V_-:=(V,\bar \partial_-^V)$ is given by the non-trivial extension
$$0\rightarrow \mathcal{O}_-\xrightarrow {\chi_-}V_-\rightarrow \mathcal{O}_-(m_-)\rightarrow 0$$
for $m_-\in\mathbb{Z}^{\geq2}$.
\end{itemize} 
We assume the images of $\mathcal{O}_{\pm}$ in $V_{\pm}$ coincide as smooth line subbundles of $V$. Then $\widetilde{L}:=(\mathbb{O},\bar\partial_{0,+},\bar\partial^\prime_{0,-})$ is the only $I_{\pm}$-holomorphic line subbundle of $(V, \bar \partial^V_+, \bar \partial^V_-)$ \cite{Hu}.  The $I_\pm$-holomorphic $Q$-bundle $\mathcal{E^\prime}=(E^\prime,\phi^\prime)$ is given by $E^\prime_i=\widetilde{L},E^\prime_j=(V, \bar \partial^V_+, \bar \partial^V_-)$ and $\phi^\prime_a$ is determined by the inclusions $\chi_\pm$. To find the constraints on stability parameters, note that $\mathcal{E}'$ has 3 proper $Q$-subbundles:
\begin{itemize}
\item[(i)] $\mathcal{F}=(F,\phi)$, where $F_i=\tilde{L}$, $F_j=\tilde{L}$ and $\phi_a$ is induced by $\phi_a'$, which is identity;
\item[(ii)] $\mathcal{F}=(F,\phi)$, where $F_i=0$, $F_j=\tilde{L}$ and $\phi_a=0$;
\item[(iii)] $\mathcal{F}=(F,\phi)$, where $F_i=0$, $F_j=(V, \bar \partial^V_+, \bar \partial^V_-)$ and $\phi_a=0$,
\end{itemize} 
therefore, $\mathcal{E^\prime}$ is $(\alpha,\sigma,\tau)$-stable if and only if the following inequalities are satisfied
      \begin{align*}
        (\sigma_i+\sigma_j)\sigma_j(\alpha_j m_+-(1-\alpha_j)m_-)&<\sigma_j\tau_i-\sigma_i\tau_j,  \\
                         \sigma_j^2(\alpha_j m_+-(1-\alpha_j)m_-)&<-\sigma_j\tau_i+\sigma_i\tau_j,  \\
                          \sigma_i\sigma_j(\alpha_j m_+-(1-\alpha_j)m_-)&>2(\sigma_j\tau_i-\sigma_i\tau_j).
      \end{align*}
               \end{enumerate}
\end{example}

\section{The Hitchin--Kobayashi Correspondence}

\begin{lemma}\label{lem3.1}
If there exists an  $(\alpha,\sigma,\tau)$-Hermitian--Einstein metric on an $I_\pm$-holomorphic $Q$-bundle $\mathcal{E}=(E,\phi)$  over an  $n$-dimensional generalized K\"{a}hler manifold $(X, I_+,I_-,g, b)$, then  $\mathcal{E}$ is $(\alpha,\sigma,\tau)$-polystable.
\end{lemma}
\begin{proof} Let $\mathcal{E}^\prime$ be a proper coherent $Q$-subsheaf of $\mathcal{E}$.  At each vertex $i\in Q_0$, one defines the   orthogonal projections $p_{\pm i}: E_{\pm i}\rightarrow E^\prime_{\pm i}$, which are defined outside $S_{\pm i}$, respectively, via the metric  $H_i$, then we have
\begin{align*}
   \deg_\pm(E_{\pm i}^\prime)=\frac{(n-1)!}{2\pi}\int_{X\backslash S_{\pm i}}[\Tr(\sqrt{-1}p_{\pm i}\circ \Lambda_{\pm}\mathbb{F}^{\pm}_{H_i})-|\xi_i^{\pm}|^2_{H_i}]d\Vol_g,
\end{align*}
where $\xi_i^\pm=\bar \partial_{\pm i}p_{\pm i}$ denote the second fundamental forms which are of class  $L^2$. Hence, by assumption that $H$ is a Hermitian--Einstein metric on $\mathcal{E}$, the degree is calculated as
\begin{align*}
  \deg_{\alpha,\sigma,\tau}(\mathcal{E}^\prime)=&\mu_{\alpha,\sigma,\tau}(\mathcal{E})\sum_{i\in Q_0}\sigma_i\rk(E_i^\prime)\\
  &-\frac{(n-1)!}{2\pi}\sum_{i\in Q_0}\int_{X\backslash S}\Big(\alpha_i\sigma_i|\xi_i^{+}|^2_{H_i,g}+(1-\alpha_i)\sigma_i|\xi_i^{-}|^2_{H_i,g}\Big)d\Vol_g\\
  &-\frac{(n-1)!}{2\pi}\sum_{a\in Q_1}\int_{X\backslash S}|\phi_a^\bot|^2_{H},
\end{align*}
where $S=\bigcup\limits_{i\in Q_0}S_i$, $\phi_a^\bot$ is the composition $(E_{t(a)}^\prime)^\bot\xrightarrow{\phi_a} E_{h(a)}\xrightarrow{p_{h(a)}} E^\prime_{h(a)}$ for the orthogonal complement  $(E_{t(a)}^\prime)^\bot$ of $E_{t(a)}^\prime$ in $E_{t(a)}$ defined outside $S_{h(a)}\bigcup S_{t(a)}$, and $|\phi_a^\bot|^2_{ H}$ is defined via the induced metric $H$ on  $E_{h(a)}\otimes (E_{t(a)})^*$. It follows that $\mathcal{E}$ is semistable. Assume $\mathcal{E}=(E,\phi)$ is indecomposable, i.e. $\mathcal{E}$ cannot be written as a direct sum of two $Q$-bundles, then either $\xi_i\neq0$ for some $i\in Q_0$ or $\phi_a^\bot\neq0$ for some $a\in Q_1$, therefore $\mu_{\alpha,\sigma,\tau}(\mathcal{E}^\prime)<\mu_{\alpha,\sigma,\tau}(\mathcal{E})$, thus $\mathcal{E}$ is stable. Finally, we find that $\mathcal{E}$ is polystable.
\end{proof}

Next we will use the continuity method to show the converse direction, thus to show that if an $I_\pm$-holomorphic $Q$-bundle $\mathcal{E}=(E,\phi)$ is $(\alpha,\sigma,\tau)$-stable, then there exist an $(\alpha,\sigma,\tau)$-Hermitian--Einstein metric $H$ on it. The approach of proof we employed here mainly follows from \cite{Hu,LT}.

We fix a Hermitian metric $H$ on an $I_\pm$-holomorphic $Q$-bundle $\mathcal{E}=(E,\phi)$. If $\tilde H=Hf=\{H_i f_i\}_{i\in Q_0}$ is an
$(\alpha,\sigma,\tau)$-Hermitian--Einstein metric for $f\in S^+(H,\mathcal{E})$, then at each vertex $i$ we have
\begin{equation*}\label{2.1}
  \begin{split}
    L_{(\alpha,\sigma,\tau)i}(f):=&\,K_{(\alpha,\sigma,\tau)}(H_i)+\sqrt{-1}\Big(\alpha_i\sigma_i\Lambda_+\bar\partial_{+}
    (f_i^{-1}\partial^+_{H_i}f_i)+(1-\alpha_i)\sigma_i\Lambda_-\bar\partial_{-}(f_i^{-1}\partial^-_{H_i}f_i)\Big)\\
 &+\sum_{a\in h^{-1}(i)} \phi_a\circ (\phi_a)^{*\tilde H}-\sum_{a\in t^{-1}(i)} (\phi_a)^{*\tilde H}\circ\phi_a\\
 =&K_{(\alpha,\sigma,\tau)}(\tilde H_i)+\sum_{a\in h^{-1}(i)} \phi_a\circ (\phi_a)^{*\tilde H}-\sum_{a\in t^{-1}(i)} (\phi_a)^{*\tilde H}\circ\phi_a\\
=&\,0,
  \end{split}
\end{equation*}
where \begin{align*}
        K_{(\alpha,\sigma,\tau)}(H_i)=&\sqrt{-1}(\alpha_i\sigma_i\Lambda_+ \mathbb{F}^+_{H_i}+(1-\alpha_i)\sigma_i\Lambda_-\mathbb{F}^-_{H_i})-\lambda(\tau_i+\mu_{\alpha,\sigma,\tau}(\mathcal{E})\sigma_i)\mathrm{Id}_{E_i} \\
  (\phi_a)^{*\tilde H}=&f^{-1}_{t(a)}\circ (\phi_a)^{*H}\circ f_{h(a)}.
      \end{align*}

The perturbed equation is given by
\begin{equation*}\label{2.3}
  \begin{split}
    L_{(\alpha,\sigma,\tau)i}^\varepsilon(f):=&\,\,L_{(\alpha,\sigma,\tau)i}(f)+\varepsilon\log f_i=0
  \end{split}
\end{equation*}
for $\varepsilon\in[0,1]$.
 Consider the set
\begin{equation*}\label{2.6}
  J=\left\{
\varepsilon\in[0,1]:\text{there exists }\,f^{(\varepsilon)}\in S^+(\mathcal{E},H)\text{\ such that }  L^\varepsilon_{(\alpha,\sigma,\tau)i}(f{(\varepsilon)})=0 \text{\ holds  for each vertex\ } i\in Q_0.
\right\}.
\end{equation*}

\begin{proposition}\label{prop3.2}
\begin{enumerate}
\item[(1)] There exists a Hermitian metric $H=\{H_i\}_{i\in Q_0}$ on an $I_\pm$-holomorphic $Q$-bundle $\mathcal{E}=(E,\phi)$, such that the simultaneous equations $\{L^1_{{(\alpha,\sigma,\tau)i}}(f)=0\}_{i\in Q_0}$ has a solution $f^{(1)}\in S^+(H,\mathcal{E})$ with
                         $$\prod_{i\in Q_0}\det f_i^{(1)}=1.$$
\item[(2)] If $ \sigma_i=\sigma_j$ for all $i,j\in Q_0$, then there exists a Hermitian metric $H=\{H_i\}_{i\in Q_0}$ on an $I_\pm$-holomorphic $Q$-bundle $\mathcal{E}=(E,\phi)$, such that
\begin{align*}
                                                                                                                 \sum_{i\in Q_0} \Tr(K_{(\alpha,\sigma,\tau)}( H_i))=&0,
                                                                                                               \end{align*} and $$\prod_{i\in Q_0}\det f_i^{(\varepsilon)}=1$$ for any solution $f^{(\varepsilon)}$ of $\{L^\varepsilon_{{(\alpha,\sigma,\tau)i}}(f)=0\}_{i\in Q_0}$.
                   \end{enumerate}
 \end{proposition}

\begin{proof}(1) For any Hermitian metric $G=\{G_i\}$ on $\mathcal{E}$
 one defines the operator
$$P^{G_i}:=\alpha_i\sigma_iP_+^{G_i}+(1-\alpha_i)\sigma_iP_-^{G_i},$$
where $P_{\pm}^{G_i}:=\sqrt{-1}\Lambda_\pm\bar\partial_{\pm }\partial^\pm_{G_i}$ for each vertex $i$, and in particular, it is denoted by $P_{ i}$ when acting on functions.
Since $$\int_X\sum_{i\in Q_0}\Tr(K_{(\alpha,\sigma,\tau)}(G_{i}))d\Vol_g=0,$$ there exist a function $\chi_i$ such that
 $$P_i(\chi_i)=-\frac{1}{|Q_0|\rk(E_i)}\sum_{i\in Q_0}\Tr(K_{(\alpha,\sigma,\tau)}(G_{i})).$$
Hence by taking $\tilde H_i=e^{\chi_i}G_i$, we obtain
$$\sum_{i\in Q_0}\Tr (K_{(\alpha,\sigma,\tau)}(\tilde H_i))=0.$$
Let us define
 \begin{align*}
   f_i^{(1)}=&\exp\left({-K_{(\alpha,\sigma,\tau)}(\tilde H_i)-\sum_{a\in h^{-1}(i)} \phi_a\circ (\phi_a)^{* \tilde H}+\sum_{a\in t^{-1}(i)} (\phi_a)^{*\tilde H}\circ\phi_a}\right),\\
   H_i=&\tilde H_i(f_i^{(1)})^{-1}.
 \end{align*}
 Since  $K_{(\alpha,\sigma,\tau)}(\tilde H_i)$ is $\tilde H_i$-Hermitian for each $i\in Q_0$, and $\phi_a\circ (\phi_a)^{*\tilde H }$ is also $\tilde H_i$-Hermitian for any $a\in h^{-1}(i)$ or $a\in t^{-1}(a)$ according to the definition of $\phi_a^{*H}$, $f_i^{(1)}$ is an $H_i$-Hermitian positive-definite endomorphism. Obviously, it follows from the traceless-ness of $\sum_{i\in Q_0}K_{(\alpha,\sigma,\tau)}(\tilde H_i)$ and $\sum_{i\in Q_0}(\sum_{a\in h^{-1}(i)} \phi_a\circ (\phi_a)^{* \tilde H}-\sum_{a\in t^{-1}(i)} (\phi_a)^{*\tilde H}\circ\phi_a)$ that
 $$\prod_{i\in Q_0}\det f_i^{(1)}=1.$$
 For $H$ and $f^{(1)}$,
we have
$$L^1_{(\alpha,\sigma,\tau)i}(f^{(1)})=K_{(\alpha,\sigma,\tau)}(\tilde H_i)+\sum_{a\in h^{-1}(i)} \phi_a\circ (\phi_a)^{* \tilde H}-\sum_{a\in t^{-1}(i)}(\phi_a)^{*\tilde H}\circ\phi_a+\log f_i^{(1)}=0.$$

(2)   The metric $H$ is chosen as above, then  the following identities hold
 \begin{align*}
\sum_{i\in Q_0} \Tr(K_{(\alpha,\sigma,\tau)}( H_i))=&\sum_{i\in Q_0} \Tr\left(K_{(\alpha,\sigma,\tau)}( \tilde H_i)-P^{H_i}\Big(\log f^{(1)}_i\Big)\right)\\
=&\sum_{i\in Q_0}P_i\left(\Tr\left(\sum_{a\in h^{-1}(i)} \phi_a\circ (\phi_a)^{* \tilde H}-\sum_{a\in t^{-1}(i)} (\phi_a)^{*\tilde H}\circ\phi_a\right)\right)\\=&\frac{1}{2}\sum_{a\in Q_1}(\sigma_{h(a)}-\sigma_{t(a)})\Delta_g\Big(|\phi_a|^2_{\tilde H}\Big),
                                                                                                                \end{align*}
where we note that $P_i+P_i^{*H_i}=\sigma_i\Delta_g$ when acting on functions, here  $\Delta_g$ is the Laplacian associated to the metric $g$. As a consequence, if $\sigma_i=\sigma_j=\sigma$ for all $i,j\in Q_0$, we get
$$\sum_{i\in Q_0} \Tr(K_{(\alpha,\sigma,\tau)}( H_i))=0.$$
 Moreover, let $\{\rho_i=\Tr(\log f_i^{(\varepsilon)})\}_{i\in Q_0}$ for a solution $f^{(\varepsilon)}$ of $\{L^\varepsilon_{{(\alpha,\sigma,\tau)i}}(f)=0\}_{i\in Q_0}$, then we have
$$(|Q_0|\sigma\Delta_g+2\varepsilon)\left(\sum_{i\in Q_0} \rho_i\right)=0,$$
which implies $\sum_{i\in Q_0} \rho_i=\log(\prod_{i\in Q_0}\det f_i^{(\varepsilon)})=0$, i.e. $\prod_{i\in Q_0}\det f_i^{(\varepsilon)}=1$.
 \end{proof}

From now on we fix a background metric $H$ on $\mathcal{E}$ as in the above proposition.
 For $f\in S^+(\mathcal{E},H)$, one defines $$\hat L_{(\alpha,\sigma,\tau)i}(\varepsilon, f)= f_i\circ L_{(\alpha,\sigma,\tau)i}^\varepsilon(f).$$
 Obviously, $\hat L_{(\alpha,\sigma,\tau)i}(\varepsilon,f)\in S( E_i, H_i)$. Denote by $\mathbb{L}_{(\alpha,\sigma,\tau)i}^{\varepsilon,f}$ the linearization of $\hat L_{(\alpha,\sigma,\tau)i}(\varepsilon, f)$ , which is  calculated as follows:
\begin{align*}
  &\mathbb{L}_{(\alpha,\sigma,\tau)i}^{\varepsilon,f}(\eta)=\frac{d}{dt}|_{t=0}\hat L_{(\alpha,\sigma,\tau)i}(\varepsilon,f+t\eta)\\=&\eta_i\circ K_{(\alpha,\sigma,\tau)}(H_i)+\sqrt{-1}\eta_i\circ ( \alpha_i\sigma_i\Lambda_+\bar\partial_{+}(f_i^{-1}\partial^+_{H_i}f_i)+(1-\alpha_i)\sigma_i\Lambda_-\bar\partial_{-}(f_i^{-1}\partial^-_{H_i}f_i))\\
  &-\sqrt{-1}f_i\circ(\alpha_i\sigma_i\Lambda_+\bar\partial_+(f_i^{-1}\circ\eta_i\circ f_i^{-1}\circ\partial_{H_i}^+f_i)+(1-\alpha_i)\sigma_i\Lambda_-\bar\partial_-(f_i^{-1}\circ\eta_i\circ f_i^{-1}\circ\partial_{H_i}^-f_i))\\
  &+\sqrt{-1}f_i\circ(\alpha_i\sigma_i\Lambda_+\bar\partial_+(f_i^{-1}\partial_{H_i}\eta_i)+(1-\alpha_i)\sigma_i\Lambda_-\bar\partial_-(f_i^{-1}\partial_{H_i}\eta_i))\\
  &+\eta_i\circ\left(\sum_{a\in h^{-1}(i)} \phi_a\circ (\phi_a)^{*\tilde H}-\sum_{a\in t^{-1}(i)} (\phi_a)^{*\tilde H}\circ\phi_a\right)\\
  &+f_i\circ\left(\sum_{a\in h^{-1}(i)} \phi_a\circ f^{-1}_{t(a)}\circ[\phi^{* H},\eta\circ f^{-1}]_a\circ f_{h(a)}-\sum_{a\in t^{-1}(i)}  f^{-1}_{t(a)}\circ[\phi^{* H},\eta\circ f^{-1}]_a\circ f_{h(a)}\circ\phi_a\right)\\
  &+\varepsilon\eta_i\circ\log f_i+\varepsilon f_i^{-1}\eta_i,
\end{align*}
where $\eta=(\eta_i)\in S(\mathcal{E},H)$, and $[\phi^{* H},\eta\circ f^{-1}]_a=\phi_a^{* H}\circ\eta_{h(a)}\circ f^{-1}_{h(a)}-\eta_{t(a)}\circ f^{-1}_{t(a)}\circ\phi_a^{* H}$.
It is clear that $\mathbb{L}_{(\alpha,\sigma,\tau)i}^{\varepsilon,f}$ extends to   a second-order elliptic differential operator of index zero between Sobolev spaces $L^p_k(S(\mathcal{E},H))$ and $L^p_{k-2} (S(\mathcal{E},H))$.

\begin{proposition}\label{prop3.3}
Let $\varepsilon\in (0,1]$, $\lambda\in\mathbb{R}$, $f\in S^+(\mathcal{E},H)$, $\eta\in S(\mathcal{E},H)$, and we denote $\eta_i^{f_i}=f_i^{-\frac{1}{2}}\circ\eta_i\circ f_i^{-\frac{1}{2}}$, $\phi_a^f=f_{h(a)}^{\frac{1}{2}}\circ\phi_a\circ f_{t(a)}^{-\frac{1}{2}}$. If $f$ is a solution of $\{L^\varepsilon_{(\alpha,\sigma,\tau)i}(f)=0\}_{i\in Q_0}$, and the equality $\mathbb{L}_{(\alpha,\sigma,\tau)i}^{\varepsilon,f}(\eta)+\lambda f_i\circ\log f_i=0$ holds at each vertex, then we have the following inequality
\begin{align*}
  \sum_{i\in Q_0}&\Big[P_i(|\eta_i^{f_i}|^2_{H_i})
  +\alpha_i\sigma_i(|(\partial^+_{H_i})^{f_i}\eta_i^{f_i}|_{H_i,g}^2+|\bar\partial_+^{f_i}\eta_i^{f_i}|_{H_i,g}^2)\\&+(1-\alpha_i)\sigma_i(|(\partial^-_{H_i})^{f_i}\eta_i^{f_i}|_{H_i,g}^2+|\bar\partial_-^{f_i}\eta_i^{f_i}|_{H_i,g}^2)\Big]\\
  &+2\sum_{a\in Q_1} \Big|[\phi^f, \eta^f]_a\Big|^2_{H}+2\varepsilon\sum_{i\in Q_0}\Big|\eta_i^{f_i}\Big|^2_{H_i}+2\lambda\sum_{i\in Q_0}\Big\langle\log f_i,\eta_i^{f_i}\Big\rangle_{H_i}\leq0,
\end{align*}
 where $(\partial^\pm_{H_i})^{f_i}=\mathrm{Ad}_{f_i^{-\frac{1}{2}}}\circ\partial^\pm_{H_i}\circ\mathrm{Ad}_{f_i^{\frac{1}{2}}}$ and $\bar\partial_\pm^{f_i}=\mathrm{Ad}_{f_i^{\frac{1}{2}}}\circ\bar \partial_\pm\circ\mathrm{Ad}_{f_i^{-\frac{1}{2}}}$ with the adjoint action $\mathrm{Ad}_\bullet\circ\star=\bullet\circ\star\circ\bullet^{-1}$.
\end{proposition}
\begin{proof}
Since $\mathbb{L}_{(\alpha,\sigma,\tau)i}^{\varepsilon,f}(\eta)=\eta_i\circ L^\varepsilon_{(\alpha,\sigma,\tau)i}(f)+f_i\circ\frac{d}{dt}|_{t=0}L_{(\alpha,\sigma,\tau)i}^\varepsilon(f+t\eta)$, we have
$$\frac{d}{dt}|_{t=0}L_{(\alpha,\sigma,\tau)i}^\varepsilon(f+t\eta)=-\lambda\log f_i.$$
Indeed, the left hand side has been calculated as
\begin{align*}
 &\frac{d}{dt}|_{t=0}L_{(\alpha,\sigma,\tau)i}^\varepsilon(f+t\eta)\\=&-\sqrt{-1}\left(\alpha_i\sigma_i\Lambda_+\bar\partial_+\Big(f_i^{-1}\circ\eta_i\circ f_i^{-1}\circ\partial_{H_i}^+f_i\Big)+(1-\alpha_i)\sigma_i\Lambda_-\bar\partial_-\Big(f_i^{-1}\circ\eta_i\circ f_i^{-1}\circ\partial_{H_i}^-f_i\Big)\right)\\
  &+\sqrt{-1}\Big(\alpha_i\sigma_i\Lambda_+\bar\partial_+\Big(f_i^{-1}\partial_{H_i}^+\eta_i\Big)+(1-\alpha_i)\sigma_i\Lambda_-\bar\partial_-\Big(f_i^{-1}\partial_{H_i}^-\eta_i\Big)\Big)\\
  &+\sum_{a\in h^{-1}(i)} \phi_a\circ f^{-1}_{t(a)}\circ\Big[\phi^{* H},\eta\circ f^{-1}\Big]_a\circ f_{h(a)}-\sum_{a\in t^{-1}(i)}  f^{-1}_{t(a)}\circ\Big[\phi^{* H},\eta\circ f^{-1}\Big]_a\circ f_{h(a)}\circ\phi_a\\
  &+\varepsilon \frac{d}{dt}|_{t=0}\log(f_i+t\eta_i)\\
  =&f_i^{-\frac{1}{2}}\circ \left(P^{H_i}_{f_i}\Big(\eta_i^{f_i}\Big)+\sum_{a\in h^{-1}(i)} \phi_a^f \circ\Big[ (\phi^f)^{*H}, \eta^f \Big]_a-\sum_{a\in t^{-1}(i)}\Big[ (\phi^f)^{*H}, \eta^f \Big]_a\circ\phi_a^f+\varepsilon\Phi \right)\circ f_i^{\frac{1}{2}},
\end{align*}
where $P^{H_i}_{f_i}=\alpha_i\sigma_iP^{H_i}_{+f_i}+(1-\alpha_i)\sigma_iP^{H_i}_{-f_i}$ for $P^{H_i}_{\pm f_i}=\sqrt{-1}\Lambda_\pm\bar\partial_\pm^{f_i}(\partial^\pm_{H_i})^{f_i}$ and  $\Phi=\mathrm{Ad}_{f_i^{\frac{1}{2}}}\circ\frac{d}{dt}|_{t=0}\log(f_i+t\eta_i)$.
Taking inner product with $\eta_i^{f_i}$ and then taking sum over all vertices on both sides, we have
\begin{align*}
  \sum_{i\in Q_0}&\left[P_i\big(\big|\eta_i^{f_i}\big|^2_{H_i}\big)+\alpha_i\sigma_i\Big(\big|\big(\partial^+_{H_i}\big)^{f_i}\eta_i^{f_i}\big|_{H_i,g}^2+\big|\bar\partial_+^{f_i}\eta_i^{f_i}\big|_{H_i,g}^2\Big)\right.\\
  &\left.+(1-\alpha_i)\sigma_i\Big(\big|\big(\partial^-_{H_i}\big)^{f_i}\eta_i^{f_i}\big|_{H_i,g}^2+\big|\bar\partial_-^{f_i}\eta_i^{f_i}\big|_{H_i,g}^2\Big)\right]\\
  &+2\sum_{a\in Q_1}\big|\big[\phi^f,\eta^f\big]_a\big|^2_H+2\varepsilon\sum_{i\in Q_0} \big\langle\Phi,\eta_i^{f_i}\big\rangle_{H_i}+2\lambda \sum_{i\in Q_0}\big\langle\log f_i,\eta_i^{f_i}\big\rangle_{H_i}=0.
\end{align*}
The estimate $\big\langle\Phi,\eta_i^{f_i}\big\rangle_{H_i}\geq \big|\eta_i^{f_i}\big|^2_{H_i}$ leads to the desired inequality.
\end{proof}

\begin{proposition}\label{prop3.4}
$J$ is a non-empty open subset of $(0,1]$.
\end{proposition}
\begin{proof}This claim is an application of implicit function theorem for Banach spaces. We only need to show the operator $\mathbb{L}_{(\alpha,\sigma,\tau)}^{\varepsilon,f}:=\bigoplus_{i\in Q_0}\mathbb{L}_{(\alpha,\sigma,\tau)i}^{\varepsilon,f}:L^p_k(S(H,\mathcal{E}))\rightarrow L^p_{k-2}(S(H,\mathcal{E}))$ sending $\eta\in L^p_k(S(H,\mathcal{E}))$ to $(\mathbb{L}_{(\alpha,\sigma,\tau)i}^{\varepsilon,f}(\eta))$ is injective. Indeed, assume $\mathbb{L}_{(\alpha,\sigma,\tau)i}^{\varepsilon,f}(\eta)=0$, then since $P^{H_i}$ is a positive operator and $\varepsilon,\alpha_i, 1-\alpha_i,\sigma_i>0$, we find that  $|\eta_i^{f_i}|^2_{H_i}=0$ for $\forall i\in Q_0$ due to maximal principle. Therefore,  it follows that  $\eta=0$ from the above proposition (put $\lambda=0$). Moreover, one can show that any solution $f\in L_k^p(S^+(H,\mathcal{E}))$ is in fact smooth by some rather standard arguments.
\end{proof}

Let $\varepsilon_0\in(0,1]$ and suppose there exists  a solution  $f^{(\varepsilon)}\in S^+(\mathcal{E},H)$ of  $\{L^\varepsilon_{(\alpha,\sigma,\tau)i}(f)=0\}_{i\in Q_0}$ for any $\varepsilon>\varepsilon_0>0$ with $\varepsilon\in(0,1]$.
From the proof of Proposition \ref{prop3.2}, we may assume $\int_X\log(\prod_{i\in Q_0}\det f_i^{(\varepsilon)})d\Vol_g=0$.
Let $\mu^{(\varepsilon)}=\frac{df^{(\varepsilon)}}{d\varepsilon}$, $\nu^{(\varepsilon)}=(f^{(\varepsilon)})^{-\frac{1}{2}}\circ\mu^{(\varepsilon)}\circ(f^{(\varepsilon)})^{-\frac{1}{2}}$, $\gamma^{(\varepsilon)}=\mathrm{Ad}_{(f^{(\varepsilon)})^{-\frac{1}{2}}}\circ\nu^{(\varepsilon)}$, and $m^{(\varepsilon)}=(m_i^{(\varepsilon)})$ for $m_i^{(\varepsilon)}=\max\limits_X\{|\log f_i^{(\varepsilon)}|\}$.
Sometimes for convenience,  we will drop the upper index $(\varepsilon)$ when there is no ambiguity.

\begin{proposition}\label{prop3.5}
Let $\mathcal{E}=(E,\phi)$ be a simple $I_\pm$-holomorphic $Q$-bundle. Then there exist  positive constant $C(m), D(m)$ depending only on $m$ such that we have   the following inequalities
\begin{align*}
 (1)\ \ \ &\sum\limits_{i\in Q_0}(\alpha_i\sigma_i||\bar\partial_+^{f_i} \nu_i||_{L^2}^2+(1-\alpha_i)\sigma_i||\bar\partial_-^{f_i} \nu_i||_{L^2}^2)+\sum\limits_{a\in Q_1}||[\phi^f,\nu]_a||_{L^2}^2\geq C(m)\sum\limits_{i\in Q_0} ||\nu_i||_{L^2}^2,\\
 (2) \ \ \ &\max\limits_X|\mu_i|_{H_i}\leq D(m)  \text{ for any vertex }i\in Q_0.
\end{align*}
\end{proposition}

\begin{proof} (1) Firstly, we have the inequality
\begin{align*}
  &\sum_{i\in Q_0}(\alpha_i\sigma_i|\bar\partial_+^{f_i} \nu_i|^2_{H_i,g}+(1-\alpha_i)\sigma_i|\bar\partial_-^{f_i} \nu_i|^2_{H_i,g})+\sum_{a\in Q_1}|[\phi^f,\nu]_a|^2_H\\
  \geq&C_0(m)\left(\sum_{i\in Q_0}(\alpha_i\sigma_i|\bar\partial_+ \gamma_i|^2_{H_i,g}+(1-\alpha_i)\sigma_i|\bar\partial_- \gamma_i|^2_{H_i,g})+\sum_{a\in Q_1}|[\phi,\gamma]_a|^2_H\right).
\end{align*}
Integrating both sides over $X$ leads to
\begin{align*}
  &\sum_{i\in Q_0}(\alpha_i\sigma_i||\bar\partial_+^f \nu_i||_{L^2}^2+(1-\alpha_i)\sigma_i||\bar\partial_-^f \nu_i||_{L^2}^2)+\sum_{a\in Q_1}||[\phi^f,\nu]_a||_{L^2}^2\\
  \geq&C_0(m)\left(\sum_{i\in Q_0}(\alpha_i\sigma_i||\bar\partial_+ \gamma_i||_{L^2}^2+(1-\alpha_i)\sigma_i||\bar\partial_- \gamma_i||_{L^2}^2)+\sum_{a\in Q_1}||[\phi,\gamma]_a||_{L^2}^2\right).
\end{align*}
Define the operator $\Delta_i:L^p_k(S(\mathcal{E},H))\rightarrow L^p_{k-2}(S(E_i, H_i))$ as
\begin{align*}
  \Delta_i\eta=\alpha_i\sigma_i\bar\partial_+^{*H_i,g}\bar\partial_+\eta_i+(1-\alpha_i)\sigma_i\bar\partial_-^{*H_i,g}\bar\partial_-\eta_i+\sum_{a\in t^{-1}(i)}\phi^{*H}_a\circ [\phi, \eta]_a-\sum_{a\in h^{-1}(i)}[\phi, \eta]_a\circ\phi_a^{*H},
\end{align*}
for $\eta\in L^p_k(S(H,\mathcal{E}))$, and define $$\Delta=\bigoplus_{i\in Q_0}\Delta_i:L^p_k(S(H,\mathcal{E}))\rightarrow L^p_{k-2}(S(H,\mathcal{E}))$$
as $\Delta(\eta)=(\Delta_i(\eta))$.
 Then we have
$$\sum_{i\in Q_0}(\alpha_i\sigma_i|\bar\partial_+ \gamma_i|^2_{H_i,g}+(1-\alpha_i)\sigma_i|\bar\partial_- \gamma_i|^2_{H_i,g})+\sum_{a\in Q_1}|[\phi,\gamma]_a|^2_H=\langle\Delta(\gamma),\gamma\rangle_H.$$
Obviously, $\Delta$ is an elliptic self-adjoint positive operator, and the assumption that $\mathcal{E}$ is simple implies $\Ker \Delta=\mathbb{C }\mathrm{Id}_\mathcal{E}$.  Therefore,
$$\langle\langle\Delta(\gamma),\gamma\rangle\rangle_{L^2}\geq \kappa ||\gamma||^2_{L^2}=\kappa\sum_{i\in Q_0}||\gamma_i||^2_{L^2},$$
where $\kappa$ is the smallest positive eigenvalue of $\Delta$ because we have the restriction
$$\sum_{i\in Q_0}\int_X\Tr(\gamma_i)d\Vol_g=\frac{d}{d\varepsilon}\int_X\log\left(\prod_{i\in Q_0}\det f_i\right)d\Vol_g=0.$$
The desired inequality is then obtained by $\sum_{i\in Q_0}|\gamma_i|^2_{H_i}\geq C_1(m)\sum_{i\in Q_0}|\nu_i|^2_{H_i}$.

(2) Since $\frac{d}{d\varepsilon}\hat L_{(\alpha,\sigma,\tau)i}(\varepsilon,f)=\mathbb{L}_{(\alpha,\sigma,\tau)i}^{\varepsilon,f}(\mu)+f_i\circ\log f_i=0$, applying Proposition \ref{prop3.3}, we have
\begin{align*}
  &\sum_{i\in Q_0}\left[\alpha_i\sigma_i||\bar\partial_+^{f_i}\nu_i||_{L^2}^2+(1-\alpha_i)\sigma_i||\bar\partial_-^{f_i}\nu_i||_{L^2}^2\right]
  +\sum_{a\in Q_1}||[\phi^f,\nu]_a||^2_{L^2}\\ \leq &2\sum_{i\in Q_0}||\log f_i||_{L^2}||\nu_i||_{L^2}\leq C_2(m)\sum_{i\in Q_0}||\nu_i||_{L^2}.
\end{align*}
The inequality in (1) gives rise to $$C(m) \sum_{i\in Q_0}||\nu_i||^2_{L^2}\leq C_2(m)\sum_{i\in Q_0}||\nu_i||_{L^2},$$ hence  $||\nu_i||_{L^2}\leq C_3(m)$  for any vertex $i\in Q_0$.
On the other hand, again by Proposition \ref{prop3.3}, there exists a second-order elliptic operator $P$ such that
$$P\left(\sum_{i\in Q_0}|\nu_i|^2_{H_i}\right)\leq \sum_{i\in Q_0}P_i(|\nu_i|^2_{H_i})\leq\sum_{i\in Q_0}(|\nu_i|^2_{H_i}+m_ i^2),$$
which implies
$$\sum_{i\in Q_0}\max_X|\nu_i|^2_{H_i}\leq C_4\left(\int_X\sum_{i\in Q_0}|\nu_i|^2_{H_i}d\Vol_g+\sum_{i\in Q_0}m_ i^2\right)\leq C_5(m),$$
thus $\max\limits_X|\mu_i|_{H_i}\leq D(m)$  for any vertex $i\in Q_0$.
\end{proof}

\begin{proposition}\label{prop3.6}
For any $f\in S^+(\mathcal{E},H)$, we have the following inequalities
 \begin{align*}
  {(1)} \ \ \     &\sum_{i\in Q_0}\left\langle\sum_{a\in h^{-1}(i)} \phi_a\circ f_{t(a)}^{-1} \circ(\phi_a)^{*H}\circ f_{i}-\sum_{a\in t^{-1}(i)}f_{i}^{-1} \circ(\phi_a)^{*H}\circ f_{h(a)}\circ\phi_a,\log f_i\right\rangle_{H_i}\\
  \geq &\sum_{i\in Q_0}\left\langle\sum_{a\in h^{-1}(i)} \phi_a(\phi_a)^{*H}-\sum_{a\in t^{-1}(i)}(\phi_a)^{*H}\circ\phi_a,\log f_i\right\rangle_{H_i}\\
  {(2)} \ \ \      &\left\langle\sqrt{-1}(\alpha_i\sigma_i\Lambda_+\bar\partial_{+}
    (f_i^{-1}\partial^+_{H_i}f_i)+(1-\alpha_i)\sigma_i\Lambda_-\bar\partial_{-}(f_i^{-1}\partial^-_{H_i}f_i)),\log f_i\right\rangle_{H_i}\geq \frac{1}{2}P_i(|\log f_i|^2_{H_i}).
        \end{align*}
\end{proposition}

\begin{proof}(1) At each point of $X$, we write
\begin{align*}
 f_i=&\sum_{A_i=1}^{r_i}\exp{(\theta_{A_i})}e_{ A_i}\otimes e^{A_i},\\
 \phi_a=&\sum_{A_{h(a)}=1}^{r_{h(a)}}\sum_{B_{t(a)}=1}^{r_{t(a)}}(\phi_a)_{A_{h(a)}}^{B_{t(a)}}e_{A_{h(a)}}\otimes e^{B_{t(a)}},
\end{align*}
where $r_i=\rk(E_i)$, $\{e_{A_i}\}_{A_i}$ forms  a $H_i$-unitary frame of $E_i$ and $\{e^{A_i}\}_{A_i}$ stands for the dual frame, and $\theta_{A_i}$'s  are real numbers.
Then we calculate pointwisely
\begin{align*}
 &\sum_{i\in Q_0}\left\langle\sum_{a\in h^{-1}(i)} \phi_a\circ f_{t(a)}^{-1} \circ(\phi_a)^{*H}\circ f_{i}-\sum_{a\in t^{-1}(i)}f_{i}^{-1} \circ(\phi_a)^{*H}\circ f_{h(a)}\circ\phi_a,\log f_i\right\rangle_{H_i}\\
=&\sum_{a\in Q_1}\sum_{A_{h(a)}=1}^{r_{h(a)}}\sum_{B_{t(a)}=1}^{r_{t(a)}}\exp(\theta_{A_{h(a)}}-\theta_{B_{t(a)}})(\theta_{A_{h(a)}}-\theta_{B_{t(a)}})|(\phi_a)_{A_{h(a)}}^{B_{t(a)}}|^2\\
\geq&\sum_{a\in Q_1}\sum_{A_{h(a)}=1}^{r_{h(a)}}\sum_{B_{t(a)}=1}^{r_{t(a)}}(\theta_{A_{h(a)}}-\theta_{B_{t(a)}})|(\phi_a)_{A_{h(a)}}^{B_{t(a)}}|^2\\
=&\sum_{i\in Q_0}\left\langle\sum_{a\in h^{-1}(i)} \phi_a\circ(\phi_a)^{*H}-\sum_{a\in t^{-1}(i)}(\phi_a)^{*H}\circ\phi_a,\log f_i\right\rangle_{H_i}.
\end{align*}

(2) It is similar to Lemma 5.14 in \cite{Hu}.\end{proof}

\begin{corollary}\label{cor3.7}
Assume $f\in S^+(H,\mathcal{E})$ is a solution of $\{ L_{(\alpha,\sigma,\tau)i}^\varepsilon(f)=0\}_{i\in Q_0}$. Then 
\begin{itemize}
\item[(1)] $m_i\leq \frac{1}{\varepsilon}m_K$, where $m_K:=|Q_0|\underset{i\in Q_0}{\sum}\underset{X}{\max}|K^\phi_{(\alpha,\sigma,\tau)}(H_i)|_{H_i}$.
\item[(2)]There exist positive constants $ C, C^\prime$ independent of $m$ such that for any vertex $i\in Q_0$ we have the inequality
$$m_i\leq C\left(\underset{i\in Q_0}{\max}\{||\log f_i||_{L^2}\}+C^\prime\right).$$
\end{itemize}
\end{corollary}

\begin{proof}(1) Since $ L_{(\alpha,\sigma,\tau)i}^\varepsilon(f)=0$, we have
\begin{align*}
 \varepsilon \sum_{i\in Q_0} |\log f_i|^2_{H_i}=&-\sum_{i\in Q_0}\bigg\langle K_{(\alpha,\sigma,\tau)}(H_i)+\sqrt{-1}\Big(\alpha_i\sigma_i\Lambda_+\bar\partial_{+}
    (f_i^{-1}\partial^+_{H_i}f_i)+(1-\alpha_i)\sigma_i\Lambda_-\bar\partial_{-}(f_i^{-1}\partial^-_{H_i}f_i)\Big)\\
 &\ \ \ \ \ \ \ \ \ \ \ \ \ \ +\sum_{a\in h^{-1}(i)} \phi_a\circ (\phi_a)^{*\tilde H}-\sum_{a\in t^{-1}(i)} (\phi_a)^{*\tilde H}\circ\phi_a,\log f_i\bigg\rangle_{H_i}\\
 \leq &-\sum_{i\in Q_0}\Big(\langle K^\phi_{(\alpha,\sigma,\tau)}(H_i),\log f_i\rangle_{H_i}+\frac{1}{2}P_i(|\log f_i|^2_{H_i})\Big)\\
 \leq&\sum_{i\in Q_0}\Big(|K^\phi_{(\alpha,\sigma,\tau)}(H_i)|_{H_i}\Big)\left(\sum_{j\in Q_0}|\log f_j|_{H_j}\right)-\frac{1}{2}\sum_{i\in Q_0}P_i(|\log f_i|^2_{H_i}),
\end{align*}
where $K^\phi_{(\alpha,\sigma,\tau)}(H_i)=K_{(\alpha,\sigma,\tau)}(H_i)+\sum\limits_{a\in h^{-1}(i)} \phi_a\circ (\phi_a)^{* H}-\sum\limits_{a\in t^{-1}(i)} (\phi_a)^{* H}\circ\phi_a$. Therefore we arrive at
$$\varepsilon \max_{i\in Q_0}\{m_i^2\}\leq |Q_0| \left(\sum_{i\in Q_0}\max\limits_X|K^\phi_{(\alpha,\sigma,\tau)}(H_i)|_{H_i}\right)\max_{i\in Q_0}\{m_i\}, $$
which implies
$$m_i\leq \max_{i\in Q_0}\{m_i\}\leq \frac{1}{\varepsilon}m_K$$
for $m_K:=|Q_0|\underset{i\in Q_0}{\sum}\underset{X}{\max}|K^\phi_{(\alpha,\sigma,\tau)}(H_i)|_{H_i}$.

(2) We have seen that
$$ \sum_{i\in Q_0}P_i(|\log f_i|^2_{H_i})\leq 2\sum_{i\in Q_0}|K^\phi_{(\alpha,\sigma,\tau)}(H_i)|_{H_i}|\log f_i|_{H_j}\leq \sum_{i\in Q_0}\Big(|K^\phi_{(\alpha,\sigma,\tau)}(H_i)|^2_{H_i}+|\log f_i|^2_{H_i}\Big).$$
As in Proposition \ref{prop3.5} (2), we get the desired inequalities.
\end{proof}

\begin{proposition}\label{prop3.8}Let $\mathcal{E}=(E,\phi)$ be a simple $I_\pm$-holomorphic $Q$-bundle. Assume there is a smooth family solution $f^{(\varepsilon)}\in S^+(\mathcal{E},H)$ to $\{L^\varepsilon_{(\alpha,\sigma,\tau)i}(f)=0\}_{i\in Q_0}$ and that there is a uniform $m^\prime$ so that $m^{(\varepsilon)}<m^\prime$ for all $\varepsilon\in (\varepsilon_0,1]$. Then there exits a constant $C(m^\prime)$ independent of $\varepsilon$ such that  $||f_i^{(\varepsilon)}||_{L_2^p}\leq C(m^\prime)$ for each vertex $i\in Q_0$.
\end{proposition}
\begin{proof}By K\"{a}hler identities on Gauduchon manifold (Lemma 7.2.5 in \cite{LT}), we have
\begin{align*}
 \Delta_i(\mu)=&P^{H_i}(\mu_i)+(P^{H_i})^{* H_i}(\mu_i)-[K_{(\alpha,\sigma,\tau)}(H_i),\mu_i]\\&+\sum_{a\in t^{-1}(i)}\phi_a^{*H}\circ [\phi, \mu]_a-\sum_{a\in h^{-1}(i)}[\phi, \mu]_a\circ\phi_a^{*H}.
\end{align*}
  Since $\Delta+\mathrm{Id}$ is self-adjoint and has strictly positive spectrum and by Proposition \ref{prop3.5} (2),  there is a positive constant $C$ such that
 \begin{align*}
   \sum_{i\in Q_0}||\mu_i||_{L_2^p}\leq &C\sum_{i\in Q_0}||\mu_i+\Delta_i(\mu_i)||_{L^p}\\
   \leq & C\sum_{i\in Q_0}\Big(||\mu_i||_{L^p}+||P^{H_i}(\mu_i)||_{L^p}+||(P^{H_i})^{*H_i}(\mu_i)||_{L^p}\Big)\\
   \leq & C_0(m^\prime)\left(1+\sum_{i\in Q_0}\Big(||\mu_i||_{L_1^{2p}}+||P^{H_i}(\mu_i)||_{L^p}\Big)\right).
 \end{align*}
On the other hand, it follows from $\mathbb{L}_{(\alpha,\sigma,\tau)i}^{\varepsilon,f}(\mu)+f_i\circ\log f_i=0$ and $L^\varepsilon_{(\alpha,\sigma,\tau)i}(f)=0$ that the variation $\mu_i$ satisfies the equation
\begin{align*}
P^{H_i}(\mu_i)=&\sqrt{-1}\mu_i\circ\Big(\alpha_i\sigma_i\Lambda_+\bar \partial_+(f_i^{-1}\circ \partial^+_{H_i}f_i)+(1-\alpha_i)\sigma_i\Lambda_-\bar \partial_-(f_i^{-1}\circ \partial^-_{H_i}f_i)\Big)\\
&+\sqrt{-1}\Big(\alpha_i\sigma_i\Lambda_+(\bar \partial_+\mu_i\circ f^{-1}_i\circ \partial^+_{H_i}f_i-\bar \partial_+ f_i\circ f^{-1}_i\circ\mu_i\circ \partial^+_{H_i}f_i+\bar \partial_+f_i\circ f^{-1}_i\circ \partial^+_{H_i}\mu_i)\\
&+\alpha_i(1-\sigma_i)\Lambda_-(\bar \partial_-\mu_i\circ f^{-1}_i\circ \partial^-_{H_i}f_i-\bar \partial_- f_i\circ f^{-1}_i\circ\mu_i\circ \partial^-_{H_i}f_i+\bar \partial_-f_i\circ f^{-1}_i\circ \partial^-_{H_i}\mu_i)\Big)\\
&-\sum_{a\in h^{-1}(i)} \phi_a\circ f^{-1}_{t(a)}\circ[\phi^{* H},\mu\circ f^{-1}]_a\circ f_{h(a)}+\sum_{a\in t^{-1}(i)}  f^{-1}_{t(a)}\circ[\phi^{* H},\mu\circ f^{-1}]_a\circ f_{h(a)}\circ\phi_a\\
&-f_i\circ\left(\varepsilon \frac{d}{dt}|_{t=0}\log(f_i+t\mu_i)+\log f_i\right).
\end{align*}
We need to estimate the $L^p$-norms of the terms on the right hand side (cf.  Proposition 3.3.5 in \cite{LT}). Indeed, we have \begin{itemize}                                                                                                                \item the first term is exactly $-\sqrt{-1}\mu_i\circ(K^\phi_{(\alpha,\sigma,\tau)}(H_i)+\varepsilon\log f_i)$, hence the norm is bounded by a constant $C_1(m^\prime)$ thanks to Proposition \ref{prop3.5} (2);                                                                                                               
\item   the norm of  second term is bounded by $C_2(m^\prime)(||\mu_i||_{L_1^{2p}}||f_i||_{L_1^{2p}}+||f_i||^2_{L_1^{2p}})$ due to H\"{o}lder's inequality;                                                                                                                    \item the norms of the third and the forth terms are obviously bounded by  some constants $C_3(m^\prime), C_4(m^\prime)$;                                                                                                                
\item the norm of the last term  is also bounded by a constant $C_5(m^\prime)$ since 
$$\left|\frac{d}{dt}\right|_{t=0}\log(f_i+t\mu_i)|_{H_i}\leq |\mu_i\circ f_i^{-1}|_{H_i}.$$                                                                                                   \end{itemize} 
                                                                                                              Consequently, we obtain
\begin{align*}
  ||P^{H_i}(\mu_i)||_{L^p}\leq C_6(m^\prime)(1+||\mu_i||_{L_1^{2p}}||f_i||_{L_1^{2p}}+||f_i||^2_{L_1^{2p}}),
\end{align*}
hence,
\begin{align*}
  \sum_{i\in Q_0}||\mu_i||_{L_2^p}\leq C_7(m^\prime)\left(1+ \sum_{i\in Q_0}||\mu_i||_{L_1^{2p}}\Big(1+\sum_{j\in Q_0}||f_j||_{L_1^{2p}}\Big)+\sum_{i\in Q_0}||f_i||_{L_1^{2p}}\sum_{j\in Q_0}||f_j||_{L_1^{2p}}\right).
\end{align*}
When $\sum\limits_{i\in Q_0}||\mu_i||_{L_2^p}\geq |Q_0|$, we have
$ \sum\limits_{i\in Q_0}||\mu_i||_{L_1^{2p}}\leq C_8(m^\prime)\sum\limits_{i\in Q_0}||\mu_i||^{\frac{1}{2}}_{L_2^{p}}$ by an interpolation inequality of Aubin. We may assume  $\sum\limits_{i\in Q_0}||f_i||_{L_2^p}\geq |Q_0|$, otherwise the conclusion has already holds truly, then similarly $\sum\limits_{i\in Q_0}||f_i||_{L_1^{2p}}\leq C_{9}(m)\sum\limits_{i\in Q_0}||f_i||^{\frac{1}{2}}_{L_2^{p}}$. Therefore we arrive at
\begin{align*}
\sum_{i\in Q_0}||\mu_i||_{L_2^p}\leq &C_{7}(m^\prime)\left(1+C_8(m^\prime)\sum\limits_{i\in Q_0}||\mu_i||^{\frac{1}{2}}_{L_2^{p}}\Big(1+C_9(m^\prime)\sum\limits_{j\in Q_0}||f_j||^{\frac{1}{2}}_{L_2^{p}}\Big)+\Big(C_9(m^\prime)\sum\limits_{i\in Q_0}||f_i||^{\frac{1}{2}}_{L_2^{p}}\Big)^2\right)\\
\leq& C_{10}(m^\prime)\left(1+\sum\limits_{i\in Q_0}||\mu_i||^{\frac{1}{2}}_{L_2^{p}}\sum\limits_{j\in Q_0}||f_j||^{\frac{1}{2}}_{L_2^{p}}+\Big(\sum\limits_{i\in Q_0}||f_i||^{\frac{1}{2}}_{L_2^{p}}\Big)^2\right).
\end{align*}
Let $\underline{\mu}=\max_{i\in Q_0}\{||\mu_i||_{L_2^p}\}$, $\underline{f}=\max_{i\in Q_0}\{||f_i||_{L_2^p}\}$, then
$$\underline{\mu}\leq C_{11}(m^\prime)\left(1+\underline{\mu}^{\frac{1}{2}}\underline{f}^{\frac{1}{2}}+\underline{f}\right),$$
which implies for each vertex $i\in Q_0$, we have
$$||\mu_i||_{L_2^p}\leq \underline{\mu}\leq C_{12}(m^\prime)(1+\underline{f}).$$
Clearly, the above inequality is also satisfied when $\sum\limits_{i\in Q_0}||\mu_i||_{L_2^p}\leq |Q_0|$.
Now take $i_0$ be the vertex such that $||f_{i_0}||_{L_2^p}=\underline{f}$, then
$$\frac{d}{d\varepsilon}||f_{i_0}||_{L_2^p}\geq -||\mu_{i_0}||_{L_2^p}\geq- C_{12}(m^\prime)(1+||f_{i_0}||_{L_2^p}).$$
Integration over $[\varepsilon,1]$ leads to the final inequality $||f_i||_{L_2^p}\leq \underline{f}\leq C(m^\prime)$.
\end{proof}

\begin{lemma}\label{lem3.9}
 $\mathcal{E}=(E,\phi)$ and $f^{(\varepsilon)}$ are as above. Then
\begin{enumerate}
  \item[(1)] $J=(0,1]$.
  \item[(2)] If  there is a constant $C$ such  that  $ \max\limits_{i\in Q_0}\{||f_i^{(\varepsilon)}||_{L^2}\}_i\leq C$ for all $\varepsilon\in (0,1]$ then there exists a solution $f^{(0)}$ of the equations $\{L_{(\alpha,\sigma,\tau)i}(f)=0\}_{i\in Q_0}$.
\end{enumerate}
\end{lemma}

\begin{proof} (1) Assume $J=(\varepsilon_0,1]$ for $\varepsilon_0>0$. If one shows the solution $f^{(\varepsilon)}$ actually extends to $[\varepsilon_0,1]$, contradicting with the openness of $J$, the claim follows. By Corollary \ref{cor3.7} (1), $m_i^{(\varepsilon)}\leq \frac{1}{\varepsilon}m_K< \frac{1}{\varepsilon_0} m_K:=m^\prime$, then $||f^{(\varepsilon)}_i||_{L_2^p}\leq C(m^\prime)$ as a result of Proposition \ref{prop3.8}. This uniform estimate guarantees the existence of the solution  $f^{(\varepsilon_0)}$. Choose $p>2n$. The uniform $L_2^p$ norm bound implies that there is a sequence $\varepsilon_\mathfrak{k}\rightarrow\varepsilon_0$, $f_i^{(\varepsilon_\mathfrak{k})}\rightarrow f_i^{(\varepsilon_0)}$ for each vertex $i\in Q_0$ converges weakly in $L_2^p$-norm and strongly in $L_1^p$-norm. We need to show $L^{\varepsilon_0}_{(\alpha,\sigma,\tau)i}(f^{(\varepsilon_0)})=0$. For any smooth section $\psi=(\psi_i)\in C^\infty(\End(\mathcal{E}))$, we compute
\begin{align*}
 &\langle\langle L^{\varepsilon_0}_{(\alpha,\sigma,\tau)i}(f^{(\varepsilon_0)}),\psi_i\rangle\rangle_{L^2}=\langle\langle L^{\varepsilon_0}_{(\alpha,\sigma,\tau)i}(f^{(\varepsilon_0)})-L^{\varepsilon_k}_{(\alpha,\sigma,\tau)i}(f^{(\varepsilon_k)}),\alpha\rangle\rangle_{L^2}\\
 =&\langle\langle\alpha_i\sigma_i\Lambda_+\bar\partial_{+}
    ((f_i^{(\varepsilon_0)})^{-1}\partial^+_{H_i}f^{(\varepsilon_0)}_i)+(1-\alpha_i)\sigma_i\Lambda_-\bar\partial_{-}(f_i^{(\varepsilon_0)})^{-1}\partial^-_{H_i}f^{(\varepsilon_0)}_i)\\
    &\ \ \ -\alpha_i\sigma_i\Lambda_+\bar\partial_{+}
    ((f_i^{(\varepsilon_\mathfrak{k})})^{-1}\partial^+_{H_i}f^{(\varepsilon_\mathfrak{k})}_i)-(1-\alpha_i)\sigma_i\Lambda_-\bar\partial_{-}(f_i^{(\varepsilon_\mathfrak{k})})^{-1}\partial^-_{H_i}f^{(\varepsilon_\mathfrak{k})}_i),\psi_i\rangle\rangle_{L^2}\\
    &+\langle\langle\varepsilon_0\log f_i^{(\varepsilon_0)}-\varepsilon_\mathfrak{k}\log f_i^{(\varepsilon_\mathfrak{k})},\psi_i\rangle\rangle_{L^2}\\
    &+\left\langle\left\langle\sum_{a\in h^{-1}(i)} \phi_a\circ( f^{(\varepsilon_0)}_{t(a)})^{-1}\circ (\phi_a)^{* H}\circ f^{(\varepsilon_0)}_{i}-\sum_{a\in t^{-1}(i)} (f_i^{(\varepsilon_0)})^{-1}(\phi_a)^{* H}\circ f^{(\varepsilon_0)}_{h(a)}\circ\phi_a\right.\right.\\
   &\left.\left.\ \ \ \ \ \ \ \  -\sum_{a\in h^{-1}(i)} \phi_a\circ( f^{(\varepsilon_\mathfrak{k})}_{t(a)})^{-1}\circ (\phi_a)^{* H}\circ f^{(\varepsilon_\mathfrak{k})}_{i}+\sum_{a\in t^{-1}(i)} (f_i^{(\varepsilon_\mathfrak{k})})^{-1}(\phi_a)^{* H}\circ f^{(\varepsilon_\mathfrak{k})}_{h(a)}\circ\phi_a\right\rangle\right\rangle_{L^2}.
\end{align*}
Note that the maps $\log :L_1^p\rightarrow L^2$
and $\exp: L^2\rightarrow L^2$, hence $(\bullet)^{-1}=\exp(-\log \bullet):L_1^p\rightarrow L^2$  are continuous, we then conclude that the terms on the right hand converge to zero when $\varepsilon_\mathfrak{k}\rightarrow\varepsilon_0$.

(2) This claim follows from Corollary \ref{cor3.7} (2), Proposition \ref{prop3.8} and similar arguments as above.
\end{proof}

To complete the Hitchin--Kobayashi correspondence, we should prove when the boundedness of $||f_i^{(\varepsilon)}||_{L^2}$'s is not satisfied, i.e., there is a vertex $i\in Q_0$ such that $\limsup\limits_{\varepsilon\rightarrow 0}||f_i^{(\varepsilon)}||_{L^2}=\infty$, the $I_\pm$-holomorphic $Q$-bundle $\mathcal{E}$ is not $(\alpha,\sigma,\tau)$-stable. Firstly, we observe that under such assumption on the norms, $\sum\limits_{i\in Q_0}\rk(E_i)>1$, otherwise, there is only one vertex $i$ with $f_i^{(\varepsilon)}=1$.

\begin{proposition}\label{prop3.10}
 Let $f\in S^+(\mathcal{E},H)$ be the solution  to $\{L^\varepsilon_{(\alpha,\sigma,\tau)i}(f)=0\}_{i\in Q_0}$. Then for any $0<\varsigma\leq 1$, we have the inequality
\begin{align*}
 \frac{1}{\varsigma}\sum_{i\in Q_0}P_i(\Tr f_i^\varsigma)\leq &-\sum_{i\in Q_0} \langle K^\phi_{(\alpha,\sigma,\tau)}(H_i)+\varepsilon\log f_i, f_i^\varsigma\rangle_{H_i}\\
 &-\sum_{i\in Q_0}\left(\alpha_i\sigma_i|f_i^{-\frac{\varsigma}{2}}\partial^+_{H_i}f^\varsigma|^2_{H_i}+(1-\alpha_i)\sigma_i|f_i^{-\frac{\varsigma}{2}}\partial^-_{H_i}f^\varsigma|^2_{H_i}\right)\\
 &-\sum_{a\in Q_1}\Big|f^{-\frac{\varsigma}{2}}_{t(a)}\circ[\phi^{*H},f^\varsigma]_a\Big|^2_H.
\end{align*}
\end{proposition}
\begin{proof}Firstly, we have the following equality
\begin{align*}
 -\sum_{i\in Q_0} \langle K_{(\alpha,\sigma,\tau)}(H_i), f_i^\varsigma\rangle_{H_i}=& \sum_{i\in Q_0}\bigg\langle\sqrt{-1}\Big(\alpha_i\sigma_i\Lambda_+\bar\partial_{+}
    (f_i^{-1}\partial^+_{H_i}f_i)+(1-\alpha_i)\sigma_i\Lambda_-\bar\partial_{-}(f_i^{-1}\partial^-_{H_i}f_i)\Big),f_i^\varsigma\bigg\rangle_{H_i}\\
    &+\varepsilon\sum_{i\in Q_0}\langle\log f_i, f_i^\varsigma\rangle_{H_i}+\sum_{a\in Q_1}\Big\langle f^{-1}_{t(a)}(\phi_a)^{*H}f_{h(a)},[\phi^{*H}, f^\varsigma]\Big\rangle_H.
\end{align*}
According to Lemma 5.21 in \cite{Hu}, we have
\begin{align*}
  &\bigg\langle\sqrt{-1}\Big(\alpha_i\sigma_i\Lambda_+\bar\partial_{+}
    (f_i^{-1}\partial^+_{H_i}f_i)+(1-\alpha_i)\sigma_i\Lambda_-\bar\partial_{-}(f_i^{-1}\partial^-_{H_i}f_i)\Big),f_i^\varsigma\bigg\rangle_{H_i}\\
    \geq&\frac{1}{\varsigma}P_i(\Tr f_i^\varsigma)+\alpha_i\sigma_i|f_i^{-\frac{\varsigma}{2}}\partial^+_{H_i}f_i^\varsigma|^2_{H_i}+(1-\alpha_i)\sigma_i|f_i^{-\frac{\varsigma}{2}}\partial^-_{H_i}f_i^\varsigma|^2_{H_i}.
\end{align*}
As the proof of Proposition \ref{prop3.6}, we calculate the third summand on the right side
\begin{align*}
 &\sum_{a\in Q_1}\left\langle f^{-1}_{t(a)}(\phi_a)^{*H}f_{h(a)},[\phi^{*H}, f^\varsigma]\right\rangle_H\\
 =&\sum_{a\in Q_1}\sum_{A_{h(a)}=1}^{r_{h(a)}}\sum_{B_{t(a)}=1}^{r_{t(a)}}\exp(\theta_{A_{h(a)}}-\theta_{B_{t(a)}})\Big(\exp(\varsigma\theta_{A_{h(a)}})-\exp(\varsigma\theta_{B_{t(a)}})\Big)\left|(\phi_a)_{A_{h(a)}}^{B_{t(a)}}\right|^2\\
 \geq&\sum_{a\in Q_1}\exp\Big(\varsigma(\theta_{A_{h(a)}}-\theta_{B_{t(a)}})\Big)\Big(\exp(\varsigma\theta_{A_{h(a)}})-\exp(\varsigma\theta_{B_{t(a)}})\Big)\left|(\phi_a)_{A_{h(a)}}^{B_{t(a)}}\right|^2\\
 =&\sum_{i\in Q_0}\left\langle \sum_{a\in h^{-1}(i)} \phi_a\circ(\phi_a)^{*H}-\sum_{a\in t^{-1}(i)}(\phi_a)^{*H}\circ\phi_a,f^\varsigma_i\right\rangle_{H_i}+\sum_{a\in Q_1}\Big|f^{-\frac{\varsigma}{2}}_{t(a)}\circ[\phi^{*H},f^\varsigma]_a\Big|^2_H.
\end{align*}
Combining the above two (in)equalities provides the desired result.
\end{proof}

\begin{corollary}\label{cor3.11}
With the same conditions as in Proposition \ref{prop3.10}, there is a constant $C$ independent of $\varepsilon$ such that for each vertex $i\in Q_0$, we have
$$\max_X|f_i|_{H_i}\leq C \max\limits_{i\in Q_0}\{||f_{i}||_{L^1}\},$$
\end{corollary}
\begin{proof}
Taking $\varsigma=1$, the above proposition shows
\begin{align*}
 \sum_{i\in Q_0}P_i(\Tr f_i)\leq& \sum_{i\in Q_0}\left(\max\limits_X|K^\phi_{(\alpha,\sigma,\tau)}(H_i)|_{H_i}+\varepsilon\max\limits_X|\log f_i|_{H_i}\right)|f_i|_{H_i}\\
\leq & \bigg((1+|Q_0|^2)\sum_{i\in Q_0}\max\limits_X|K^\phi_{(\alpha,\sigma,\tau)}(H_i)|_{H_i}\bigg)\left(\sum_{j\in Q_0}|f_j|_{H_j}\right)\\
\leq& C_1 \sum_{i\in Q_0}\Tr f_i
\end{align*}
where the second inequality  applies Corollary \ref{cor3.7} (1), and the third  one  follows from the inequalities  $C_2^{-1}(\Tr f_i)\leq|f_i|_{H_i}\leq C_2(\Tr f_i)$ since $f\in S^+(\mathcal{E},H)$. By the Lemma 3.2.2 in \cite{LT}, we have
$$\max_X|f_i|_{H_i}\leq C_3\max_X\left(\sum_{i\in Q_0}\Tr f_i\right)\leq C_4\left|\left|\sum_{i\in Q_0}\Tr f_i\right|\right|_{L^1}\leq C \max\limits_{i\in Q_0}\{||f_{i}||_{L^1}\},$$
as desired.
\end{proof}

For $\varepsilon>0$, $x\in X$ and $i\in Q_0$, we denote by  $e_i(\varepsilon,x)$  the largest eigenvalue of $\log f_i^{(\varepsilon)}$ for the solution $f^{(\varepsilon)}$ of $\{L^\varepsilon_{(\alpha,\sigma,\tau)i}(f)=0\}_{i\in Q_0}$ and define  $\rho_i^{(\varepsilon)}=\exp(-M_i^{(\varepsilon)})$ for $M_i^{(\varepsilon)}=\max\limits_Xe_i(\varepsilon,x)$, $\mathfrak{F}_i^{(\varepsilon)}=\rho_i^{(\varepsilon)}f_i^{(\varepsilon)}$. Let $i_0\in Q_0$ be the vertex depending on $\varepsilon$ such that $||f_{i_0}^{(\varepsilon)}||_{L^2}= \max\limits_{i\in Q_0}\{||f_{i}^{(\varepsilon)}||_{L^2}\}$.

\begin{proposition}\label{prop3.12} Assume $\limsup\limits_{\varepsilon\rightarrow 0}||f_{i_0}^{(\varepsilon)}||_{L^2}=\infty$, then there is a sequence $\varepsilon_\mathfrak{k}\rightarrow0$ such that $\rho^{(\varepsilon_\mathfrak{k})}_{i_0}\rightarrow 0$ and  $\mathfrak{F}_{i_0}^{(\varepsilon_\mathfrak{k})}$ converges weakly in $L_1^2$-norm to an $\mathfrak{F}_\infty\neq 0$.
 \end{proposition}

\begin{proof} Firstly, by definition, we have  $\rho_i^{(\varepsilon)} f_i^{(\varepsilon)}\leq \mathrm{Id}_{E_i}$ and $\max\limits_X(\rho_i^{(\varepsilon)}|f_i^{(\varepsilon)}|_{H_i})\geq 1$ for $\forall i\in Q_0$. Then  we  get
\begin{align*}
  1\leq\max\limits_X(\rho_{i_0}^{(\varepsilon)}|f_{i_0}^{(\varepsilon)}|_{H_{i_0}})\leq C_1\rho^{(\varepsilon)}_{i_0}\sum_{i\in Q_0}||f_i^{(\varepsilon)}||_{L^1}\leq C_2\rho^{(\varepsilon)}_{i_0}\sum_{i\in Q_0}||f_i^{(\varepsilon)}||_{L^2}\leq C_3||\rho^{(\varepsilon)}_{i_0}f_{i_0}^{(\varepsilon)}||_{L^2},
\end{align*}where the second inequality follows from Corollary \ref{cor3.11}.
Therefore, there is a constant $C_0$ independent of $\varepsilon$ such that
\begin{align*}
  C_0^{-1}\leq ||\rho^{(\varepsilon)}_{i_0}f_{i_0}^{(\varepsilon)}||_{L^2}\leq C_0.
\end{align*}
In particular, the above inequality on the left hand side means that if $\mathfrak{F}_{i_0}^{(\varepsilon_\mathfrak{k})}$ converges to $\mathfrak{F}_\infty$ weakly in $L_1^2$-norm then we may assume it converges strongly in $L^2$-norm, hence $\mathfrak{F}_\infty\neq 0$. On the other hand, by Corollary \ref{cor3.7} (1) and Proposition \ref{prop3.10}, we have
\begin{align*}
 \sum_{i\in Q_0}||D^+_{H_i}\rho_i^{(\varepsilon)} f_i^{(\varepsilon)}||_{L^2}&\leq 2 \sum_{i\in Q_0}\int_X\left|(\rho_i^{(\varepsilon)} f_i^{(\varepsilon)})^{-\frac{1}{2}}\partial^+_{H_i}(\rho_i^{(\varepsilon)} f_i^{(\varepsilon)})\right|^2_{H_i}d\Vol_g\\
 &\leq\frac{2}{\min\limits_{i\in Q_0}\{\alpha_i\sigma_i\}} \sum_{i\in Q_0}\int_X\left(-\big\langle K^\phi_{(\alpha,\sigma,\tau)}(H_i)+\varepsilon\log f_i^{(\varepsilon)}, \rho_i^{(\varepsilon)} f^{(\varepsilon)}_i\big\rangle_{H_i}-P_i(\Tr f_i^{(\varepsilon)})\right)d\Vol_g\\
 &\leq\frac{4}{\min\limits_{i\in Q_0}\{\alpha_i\sigma_i\}}\sum_{i\in Q_0}\max\limits_X\big|K^\phi_{(\alpha,\sigma,\tau)}(H_i)\big|_{H_i}||\rho_i^{(\varepsilon)} f_i^{(\varepsilon)}||_{L^1}\\ &\leq C_4,
\end{align*}
where $D^+_{H_i}$ is the induced Chern connection on the endomorphism bundle. This means   that $\rho^{(\varepsilon)}_{i_0}f_{i_0}^{(\varepsilon)}$ is $L_1^2$-bounded uniformly. By assumption that $L^2$-norm of $f_{i_0}^{(\varepsilon)}$ is unbounded, there is a sequence $\varepsilon_\mathfrak{k}\rightarrow 0$ such that $\rho^{(\varepsilon_\mathfrak{k})}_{i_0}\rightarrow 0$ and  $\mathfrak{F}_{i_0}^{(\varepsilon_\mathfrak{k})}$ converges weakly in $L_1^2$-norm.
\end{proof}

 We have shown that there is a sequence $\varepsilon_\mathfrak{k}\rightarrow 0$ such that $\mathfrak{F}^{(\varepsilon_\mathfrak{k})}_{i_0}$ converges weakly  to a non-zero $L_1^2$-endomorphism $\mathfrak{F}^\infty$ of $\mathcal{E}$. Similarly, for $0<\varsigma\leq 1$, define $(\mathfrak{F}_i^{(\varepsilon)})^\varsigma=(\rho_i^{(\varepsilon)} f_i^{(\varepsilon)})^\varsigma$ and $(\mathfrak{F}^{(\varepsilon)})^\varsigma=\bigoplus\limits_{i\in Q_0}(\mathfrak{F}_i^{(\varepsilon)})^\varsigma$, then there exists a sequence $\varsigma_{\mathfrak{k}}\rightarrow 0$ and $\widetilde{\mathfrak{F}^\infty}\neq 0$ such that $(\mathfrak{F}^\infty)^{\varsigma_{\mathfrak{k}}}\rightarrow\widetilde{\mathfrak{F}^\infty}$ weakly in $L_1^2$-norm.
Then we introduce the endomorphism
$$\Theta=\mathrm{Id}_\mathcal{E}-\widetilde{\mathfrak{F}^\infty}\in L_1^2(\End(\mathcal{E})).$$

\begin{proposition}\label{prop3.13}
$\Theta$ satisfies the following identities  in $L^1$-norm
\begin{enumerate}                                                                 \item[(1)] $\Theta^2=\Theta=\Theta^{*H}$,                                                                 \item[(2)] $(\mathrm{Id}_\mathcal{E}-\Theta)\circ \bar\partial_{\pm}\Theta=0$,                                                                 \item[(3)] $(\mathrm{Id}-\Theta)_{h(a)}\circ\phi_a\circ\Theta_{t(a)}=0$ for $\forall a\in Q_1$.                                                               \end{enumerate}
Therefore, $\Theta$ defines a  $Q$-coherent subsheaf $\mathcal{F}$ of $\mathcal{E}$ with the property that $$0<\sum_{i\in Q_0}\rk(F_i)<\sum_{i\in Q_0}\rk(E_i).$$
\end{proposition}
\begin{proof}(1) These facts are obvious.

(2) As the proof of Proposition 3.4.6 iii) in \cite{LT}, the following calculations lead to the identities in (2):
\begin{align*}
  &(\min\limits_{i\in Q_0}\{\alpha_i\sigma_i,(1-\alpha_i)\sigma_i,1\})\cdot(||(\mathrm{Id}_\mathcal{E}-\Theta)\circ \bar\partial_{+}\Theta||^2_{L^2}+||(\mathrm{Id}_\mathcal{E}-\Theta)\circ \bar\partial_{-}\Theta||^2_{L^2})\\
  \leq&\lim\limits_{\varsigma_2\rightarrow 0}\lim\limits_{\varsigma_1\rightarrow 0}\lim\limits_{\mathfrak{k}\rightarrow\infty}\Bigg(\sum_{i\in  Q_0}\int_X\bigg(\alpha_i\sigma_i\Big|(\mathrm{Id}_{E_i}-(\mathfrak{F}^{(\varepsilon_\mathfrak{k})}_i)^{\varsigma_1})\circ\partial^+_{H_i}(\mathfrak{F}^{(\varepsilon_\mathfrak{k})}_i)^{\varsigma_2}\Big|^2_{H_i}
  \\&\ \ \ \ \ \ \ \ \ \ \ \ \ \ \ \ \ \ \ \ \ \ \ \ \ \ +(1-\alpha_i)\sigma_i\Big|(\mathrm{Id}_{E_i}-(\mathfrak{F}^{(\varepsilon_\mathfrak{k})}_i)^{\varsigma_1})\circ\partial^-_{H_i}(\mathfrak{F}^{(\varepsilon_\mathfrak{k})}_i)^{\varsigma_2}\Big|^2_{H_i}\bigg)d\Vol_g\Bigg)\\
  \leq &\lim\limits_{\varsigma_2\rightarrow 0}\lim\limits_{\varsigma_1\rightarrow 0}\lim\limits_{\mathfrak{k}\rightarrow\infty}\left(\frac{2\varsigma_1}{2\varsigma_1+\varsigma_2}\right)^2\cdot\Bigg(\sum_{i\in  Q_0}\int_X\bigg(\alpha_i\sigma_i\Big|(\mathfrak{F}^{(\varepsilon_\mathfrak{k})}_i)^{-\frac{\varsigma_2}{2}}\circ\partial^+_{H_i}(\mathfrak{F}^{\varepsilon_\mathfrak{k}}_i)^{\varsigma_2}\Big|^2_{H_i}\\
  & \ \ \ \ \ \ \ \ \ \ \ \ \ \ \ \ \ \ \ \ \ \ \ \ \ \ \ \ \ \ \ \ \ \ \ \ \ \ \ \ \ \ \ \ \ \ \ \ +(1-\alpha_i)\sigma_i\Big|(\mathfrak{F}^{(\varepsilon_\mathfrak{k})}_i)^{-\frac{\varsigma_2}{2}}\circ\partial^-_{H_i}(\mathfrak{F}^{(\varepsilon_\mathfrak{k})}_i)^{\varsigma_2}\Big|^2_{H_i}\bigg)d\Vol_g\Bigg)\\
\leq&\lim\limits_{\varsigma_2\rightarrow 0}\lim\limits_{\varsigma_1\rightarrow 0}\lim\limits_{\mathfrak{k}\rightarrow\infty}\left(\frac{2\varsigma_1}{2\varsigma_1+\varsigma_2}\right)^2\cdot\Bigg(\sum_{i\in  Q_0}\int_X \bigg(-\frac{1}{\varsigma_2}P_i\Big(\Tr (\mathfrak{F}^{(\varepsilon_\mathfrak{k})}_i)^{\varsigma_2}\Big)\\&\ \ \ \ \ \ \ \ \ \ \ \  \ \ \ \ \ \ \ \ \ \ \ \ \ \ \ \ \ \ \ \ \ \ \ \ \ \ \ \ \ \ \ \ \  \ \ \ -\Big\langle K^\phi_{(\alpha,\sigma,\tau)}(H_i)+\varepsilon\log (f^{(\varepsilon_\mathfrak{k})}_i), (\mathfrak{F}^{(\varepsilon_\mathfrak{k})}_i)^{\varsigma_2}\Big\rangle_{H_i}\bigg)d\Vol_g\Bigg)\\
\leq &C\cdot\lim\limits_{\varsigma_2\rightarrow 0}\lim\limits_{\varsigma_1\rightarrow 0}\left(\frac{2\varsigma_1}{2\varsigma_1+\varsigma_2}\right)^2\\=&0
\end{align*}
for $0\leq \varsigma_1\leq\frac{\varsigma_2}{2}\leq 1$, where the first inequality is due to Fatou lemma, the second one is because of  the inequality
$\mathrm{Id}_{E_i}-(\mathfrak{F}^{(\varepsilon_\mathfrak{k})}_i)^{\varsigma_1}\leq \frac{2\varsigma_1}{2\varsigma_1+\varsigma_2}(\mathfrak{F}^{(\varepsilon_\mathfrak{k})}_i)^{-\frac{\varsigma_2}{2}}$, the third one is as the result of Proposition \ref{prop3.10}, and the last one follows from  Corollary \ref{cor3.7} (1).

(3) Similarly as above, since $$\Big((\mathrm{Id}-\Theta)_{h(a)}\circ\phi_a\circ\Theta_{t(a)}\Big)^{*H}=-\Theta_{t(a)}\circ[\phi^{*H},\Theta]_a=\Theta_{t(a)}\circ[\phi^{*H},\mathrm{Id}-\Theta]_a,$$ we have
\begin{align*}
&\sum_{a\in Q_1}|| (\mathrm{Id}-\Theta)_{h(a)}\circ\phi_a\circ\Theta_{t(a)}||^2_{L^2}\\
=&\lim\limits_{\varsigma_2\rightarrow 0}\lim\limits_{\varsigma_1\rightarrow 0}\lim\limits_{\mathfrak{k}\rightarrow\infty}\Bigg(\sum_{a\in Q_1}\int_X\Big|\big(\mathrm{Id}-(\mathfrak{F}^{(\varepsilon_\mathfrak{k})})^{\varsigma_1}\big)_{t(a)}\circ\big[\phi^{*H},(\mathfrak{F}^{(\varepsilon_\mathfrak{k})})^{\varsigma_2}\big]_a\Big|^2_{H}d\Vol_g\Bigg)\\
\leq&\lim\limits_{\varsigma_2\rightarrow 0}\lim\limits_{\varsigma_1\rightarrow 0}\lim\limits_{\mathfrak{k}\rightarrow\infty}\Bigg(\frac{2\varsigma_1}{2\varsigma_1+\varsigma_2}\Bigg)^2\cdot \Bigg(\sum_{a\in Q_1}\int_X\Big|\big(\mathfrak{F}^{(\varepsilon_\mathfrak{k})}_{t(a)}\big)^{-\frac{\varsigma_2}{2}}\circ\big[\phi^{*H},(\mathfrak{F}^{(\varepsilon_\mathfrak{k})})^{\varsigma_2}\big]_a\Big|^2_{H}d\Vol_g\Bigg)\\
\leq&\lim\limits_{\varsigma_2\rightarrow 0}\lim\limits_{\varsigma_1\rightarrow 0}\lim\limits_{\mathfrak{k}\rightarrow\infty}\Bigg(\frac{2\varsigma_1}{2\varsigma_1+\varsigma_2}\Bigg)^2\cdot\Bigg(\sum_{i\in  Q_0}\int_X \bigg(-\frac{1}{\varsigma_2}P_i\Big(\Tr (\mathfrak{F}^{(\varepsilon_\mathfrak{k})}_i)^{\varsigma_2}\Big)\\&\ \ \ \ \ \ \ \ \ \ \ \ \ \ \ \ \ \ \ \ \ \ \ \ \ \  \ \ \ \ \ \ \ \ \ \ \ \ \ \ \ \ \ \ \ \  \ \ \ -\Big\langle K^\phi_{(\alpha,\sigma,\tau)}(H_i)+\varepsilon\log (f^{(\varepsilon_\mathfrak{k})}_i), (\mathfrak{F}^{(\varepsilon_\mathfrak{k})}_i)^{\varsigma_2}\Big\rangle_{H_i}\bigg)d\Vol_g\Bigg)\\
= & 0,
\end{align*}
which indicates the desired identities.

The existence of $\mathcal{F}$ is due to the classical result of Uhlenbeck and Yau \cite{UY}. Non-vanishing of  $\mathfrak{F}^\infty$ implies $\sum\limits_{i\in Q_0}\rk(F_i)<\sum\limits_{i\in Q_0}\rk(E_i)$. On the other hand,  $\int_X\log(\prod\limits_{i\in Q_0}\det f_i^{(\varepsilon)})d\Vol_g=0$ means that almost everywhere either there exists a vertex $i\in Q_0$  such that $\lim\limits_{\varepsilon\rightarrow 0}|f_i^{(\varepsilon)}|_{H_i}<\infty$ and $\det f_i^{(\varepsilon)}\rightarrow 0$, or   $\lim\limits_{\varepsilon\rightarrow 0}f_i^{(\varepsilon)}$ has zero eigenvalue for some vertex $i$ with $|f_i^{(\varepsilon)}|_{H_i}\rightarrow \infty$, whenever we have $\sum\limits_{i\in Q_0}\rk(\mathfrak{F}_i^{(\varepsilon)})<\sum\limits_{i\in Q_0}\rk(E_i)$, hence $\sum\limits_{i\in Q_0}\rk(F_i)>0$.
\end{proof}

\begin{lemma}\label{lem3.14}
$\mu_{\alpha,\sigma,\tau}(\mathcal{F})\geq\mu_{\alpha,\sigma,\tau}(\mathcal{E})$.
\end{lemma}
\begin{proof}
The slope of $\mathcal{F}$ is given by
\begin{align*}
 \mu_{(\alpha,\sigma,\tau)}(\mathcal{F})=&\mu_{(\alpha,\sigma,\tau)}(\mathcal{E})+\frac{(n-1)!}{2\pi}\frac{\sum_{i\in Q_0}\int_{X\backslash S}\Tr\Big(K_{(\alpha,\sigma,\tau)}( H_i)\circ\Theta_i\Big)d\Vol_g}{\sum_{i\in Q_0}\sigma_i\rk(E_i)}\\
 &-\frac{(n-1)!}{2\pi}\frac{\sum_{i\in Q_0}\int_{X\backslash S}\Big(\alpha_i\sigma_i|\bar\partial_{+i}\Theta_i|^2_{H_i,g}+(1-\alpha_i)\sigma_i|\bar\partial_{-i}\Theta_i|^2_{H_i,g}\Big)d\Vol_g}{\sum_{i\in Q_0}\sigma_i\rk(E_i)}.
\end{align*}
Since $\sum\limits_{i\in Q_0}\int_{X\backslash S}\Tr(K^\phi_{(\alpha,\sigma,\tau)}( H_i))d\Vol_g=0$, we have
\begin{align*}
 &\sum\limits_{i\in Q_0}\int_{X\backslash S}\Tr\Big(K^\phi_{(\alpha,\sigma,\tau)}( H_i)\circ\Theta_i\Big)d\Vol_g\\
 =&-\lim\limits_{\varsigma\rightarrow0}\lim\limits_{\mathfrak{k}\rightarrow\infty}\sum_{i\in Q_0}\int_X\Tr\Big(K^\phi_{(\alpha,\sigma,\tau)}( H_i)\circ(\mathfrak{F}_i^{(\varepsilon_\mathfrak{k})})^\varsigma\Big)d\Vol_g\\
 =&\lim\limits_{\varsigma\rightarrow0}\lim\limits_{\mathfrak{k}\rightarrow\infty}\left(\sum_{i\in Q_0}\int_X\Tr\bigg(\Big(\sqrt{-1}\big(\alpha_i\sigma_i\Lambda_+\bar\partial_{+}
    ((\mathfrak{F}_i^{(\varepsilon_\mathfrak{k})})^{-1}\partial^+_{H_i}\mathfrak{F}_i^{(\varepsilon_\mathfrak{k})})\right.\\
    &\ \ \ \ \ \ \ \ \ \ \ \ \ \ \ \ \ \ \ \ \ \ \ \ \ +(1-\alpha_i)\sigma_i\Lambda_-\bar\partial_{-}((\mathfrak{F}_i^{(\varepsilon_\mathfrak{k})})^{-1}\partial^-_{H_i}\mathfrak{F}_i^{(\varepsilon_\mathfrak{k})})\big)+\varepsilon_\mathfrak{k}\log f_i^{(\varepsilon_\mathfrak{k})}\Big)\circ(\mathfrak{F}_i^{(\varepsilon_\mathfrak{k})})^\varsigma\bigg)d\Vol_g\\
    &\left.\ \ \ \ \ \ \  \ \ \ \ \ \ \ \ \ \ +\sum_{a\in Q_1}\int_X\Big\langle(\mathfrak{F}_{t(a)}^{(\varepsilon_\mathfrak{k})})^{-1}[\phi^{*H},\mathfrak{F}^{(\varepsilon_\mathfrak{k})}]_a,[\phi^{*H},\mathfrak{F}^{(\varepsilon_\mathfrak{k})}]_a\Big\rangle_Hd\Vol_g\right)\\
    \geq&\lim\limits_{\varsigma\rightarrow0}\lim\limits_{\mathfrak{k}\rightarrow\infty}\left(\sum_{i\in Q_0}\int_X\bigg(\alpha_i\sigma_i\Big|(\mathfrak{F}_i^{(\varepsilon_\mathfrak{k})})^{-\frac{\varsigma}{2}}\partial^+_{H_i}(\mathfrak{F}_i^{(\varepsilon_\mathfrak{k})})^\varsigma\Big|^2_{H_i}+(1-\alpha_i)\sigma_i\Big|(\mathfrak{F}_i^{(\varepsilon_\mathfrak{k})})^{-\frac{\varsigma}{2}}\partial^-_{H_i}(\mathfrak{F}_i^{(\varepsilon_\mathfrak{k})})^\varsigma\Big|^2_{H_i}\bigg)d\Vol_g\right.\\
 &\left.\ \ \ \ \ \ \  \ \ \ \ \ \ \ \ \ \ +\sum_{a\in Q_1}\int_X\Big\langle(\mathfrak{F}_{t(a)}^{(\varepsilon_\mathfrak{k})})^{-1}[\phi^{*H},\mathfrak{F}^{(\varepsilon_\mathfrak{k})}]_a,[\phi^{*H},\mathfrak{F}^{(\varepsilon_\mathfrak{k})}]_a\Big\rangle_Hd\Vol_g\right) \\
 \geq&\lim\limits_{\varsigma\rightarrow0}\lim\limits_{\mathfrak{k}\rightarrow\infty}\left(\sum_{i\in Q_0}\int_X\bigg(\alpha_i\sigma_i\Big|\partial^+_{H_i}\big(\mathrm{Id}_{E_i}-(\mathfrak{F}_i^{(\varepsilon_\mathfrak{k})})^\varsigma\big)\Big|^2_{H_i}+(1-\alpha_i)\sigma_i\Big|\partial^-_{H_i}\big(\mathrm{Id}_{E_i}-(\mathfrak{F}_i^{(\varepsilon_\mathfrak{k})})^\varsigma\big)\Big|^2_{H_i}\bigg)d\Vol_g\right.\\
 &\left.\ \ \ \ \ \ \  \ \ \ \ \ \ \ \ \ \ +\sum_{a\in Q_1}\int_X\Big\langle|[\phi^{*H},\mathrm{Id}-\mathfrak{F}^{(\varepsilon_\mathfrak{k})}]_a|^2_H\Big\rangle_Hd\Vol_g\right) 
 \end{align*}
    \begin{align*}
 =&\sum_{i\in Q_0}\int_{X\backslash S}\Big(\alpha_i\sigma_i|\bar\partial_{+i}\Theta_i|^2_{H_i,g}+(1-\alpha_i)\sigma_i|\bar\partial_{-i}\Theta_i|^2_{H_i,g}\Big)d\Vol_g
 +\sum_{a\in Q_1}\int_{X\backslash S}\Big|[\phi^{*H},\Theta]_a\Big|^2_Hd\Vol_g,
 \end{align*}
 where  we have noted that $$\sum\limits_{i\in Q_0}\int_X\Tr\Big(\log f_i^{(\varepsilon_\mathfrak{k})}\circ(f_i^{(\varepsilon_\mathfrak{k})})^\varsigma\Big)d\Vol_g\geq 0$$ since $\sum\limits_{i\in Q_0}\int_X\Tr(\log f_i^{(\varepsilon_\mathfrak{k})})d\Vol_g=0$ .
On the other hand, we have

\begin{align*}
  &\sum_{i\in Q_0}\int_{X\backslash S}\Tr\Big(K_{(\alpha,\sigma,\tau)}( H_i)\circ\Theta_i\Big)d\Vol_g\\=&\sum_{i\in Q_0}\int_{X\backslash S}\Tr\Big(K^\phi_{(\alpha,\sigma,\tau)}( H_i)\circ\Theta_i\Big)d\Vol_g-\sum_{a\in Q_1}\int_{X\backslash S}\Big\langle\phi^{*H}_a,[\phi^{*H},\Theta]_a\Big\rangle_Hd\Vol_g\\
  =&\sum_{i\in Q_0}\int_{X\backslash S}\Tr\Big(K^\phi_{(\alpha,\sigma,\tau)}( H_i)\circ\Theta_i\Big)d\Vol_g-\sum_{a\in Q_1}\int_{X\backslash S}\Big|[\phi^{*H},\Theta]_a\Big|^2_Hd\Vol_g,
\end{align*}
where the second equality follows from the properties that $\Theta^{*H}=\Theta$ and $(\mathrm{Id}-\Theta)_{h(a)}\circ\phi_a\circ\Theta_{t(a)}=0$ for $\forall a\in Q_1$. Putting the above calculations together confirms the lemma.
\end{proof}

Now combining Lemmas \ref{lem3.1}, \ref{lem3.9} and \ref{lem3.14} gives us the main theorem:
\begin{theorem}\label{thm3.15} Let $Q=(Q_0,Q_1)$ be a quiver, and $\mathcal{E}=(E,\phi)$ be an $I_\pm$-holomorphic $Q$-bundle over an  $n$-dimensional compact  generalized K\"{a}hler manifold $(X, I_+,I_-,g, b)$ such that $g$ is Gauduchon with respect to both $I_+$ and $I_-$, then  $\mathcal{E}$ is $(\alpha,\sigma,\tau)$-polystable if and only if $\mathcal{E}$ admits an  $(\alpha,\sigma,\tau)$-Hermitian--Einstein metric.
\end{theorem}

\paragraph{\textbf{Acknowledgments}}The author Z. Hu  would like to thank Prof. Yang-Hui He and Prof. Kang Zuo for their useful discussions, and the author P. Huang would like to thank Prof. Jiayu Li and Prof. Xi Zhang for their  kind help and encouragement. The author P. Huang was financial supported by China Scholarship Council (No. 201706340032). The authors would like to thank the referees for their careful reading.

\end{document}